\documentclass[11pt,notitlepage,twoside,a4paper]{amsart}
 \usepackage{amsfonts}
 \usepackage{mathrsfs}

\usepackage{amsmath,amssymb,enumerate}

\usepackage{epsfig,fancyhdr,color}%,showkeys,amsmidx

\usepackage{amssymb}
\usepackage{amsmath,amsthm}
\usepackage{latexsym}
\usepackage{amscd}
\usepackage{graphicx}
\usepackage[latin1]{inputenc}
\usepackage[all]{xy}

\newcommand{\T}{\mathcal{T}_{g,n}}
\newcommand{\Q}{\mathcal{QT}_{g,n}}

\newcommand{\F}{\mathcal{F}}
\newcommand{\G}{\mathcal{G}}
\newcommand{\MF}{\mathcal{MF}}
\newcommand{\PMF}{\mathcal{PMF}}
\newcommand{\MFind}{\mathcal{MF}_{ind}}

\newtheorem{theorem}{\rm\bf Theorem}[section]
\newtheorem{proposition}[theorem]{\rm\bf Proposition}
\newtheorem{lemma}[theorem]{\rm\bf Lemma}
\newtheorem{corollary}[theorem]{\rm\bf Corollary}
\newtheorem*{theorem 1}{\rm\bf Proposition 1}
\newtheorem*{theorem 2}{\rm\bf Proposition 2}

\theoremstyle{definition}
\newtheorem{definition}[theorem]{\rm\bf Definition}

\theoremstyle{remark}
\newtheorem{remark}[theorem]{\rm\bf Remark}

\def\interieur#1{\mathord{\mathop{\kern 0pt #1}\limits^\circ}}

\title[Horospheres in Teichm\"uller space]{Horospheres in Teichm\"uller space and mapping class group}
\author{Weixu Su}
\address{Weixu Su, School of Mathematics, Fudan University, 200433, Shanghai, P. R. China}
\email{suwx@fudan.edu.cn}
\author{Dong Tan}
\address{Dong Tan, Guangxi Center for Mathematical Research, Guangxi University, 530000, Nanning, P. R. China}
\email{duzuizhe2013@foxmail.com}
\date{\today}

\begin{document}

\maketitle
\begin{abstract}
We study the geometry of horospheres in
Teichm\"uller space  of Riemann surfaces of genus $g$ with $n$ punctures, where $3g-3+n\geq 2$.
We show that every $C^1$-diffeomorphism of Teichm\"uller space to itself that preserves horospheres
is an element of the extended mapping class group.
Using the relation between horospheres and metric balls,
we obtain a new proof of Royden's Theorem that
the isometry group of the Teichm\"uller metric is the extended mapping class group.
\end{abstract}

\noindent AMS Mathematics Subject Classification:   32G15; 30F30; 30F60.

\noindent Keywords: {Extremal length; horosphere; mapping class group; Teichm\"uller space.}\\

%\tableofcontents

\section{Introduction}\label{sec:into}
In this paper, we study the geometry of horospheres in Teichm\"uller space.
As an application, we give a new proof of Royden's Theorem that
every isometry of Teichm\"uller space with respect to the Teichm\"uller metric
is induced by an element of the mapping class group.
Our results rely heavily on the theory of measured foliations as found and developed in \cite{thurston1988geometry,fathi1979travaux,ivanov2001isometries,lenzhen2010criteria}.

\subsection{Background}
Let $S=S_{g,n}$ be a Riemann surface of genus $g$   with $n$ punctures,
and let $\mathcal{T}_{g,n}$ be the Teichm\"uller space of $S$.
We endow $\mathcal{T}_{g,n}$ with the Teichm\"uller metric.
Throughout this paper, we assume that $3g -3 + n \geq 2$.

Much of the study of Teichm\"uller space is inspired by analogies with negatively curved spaces.
The Teichm\"uller metric is a complete Finsler metric,
with  very rich geometry involving extremal lengths of measured foliations.
The Teichm\"uller geodesic flow and horocycle flow are
ergodic on the moduli space, with respect to the Masur-Veech measure.

Let $\mathcal{MF} = \mathcal{MF}(S)$ be the space of measured foliations on $S$.
Denote the space of projective classes in $\mathcal{MF}$ by $\mathcal{PMF}$.
Topologically, $\mathcal{PMF}$ is a sphere of dimension $6g-7+2n$.
Thurston \cite{thurston1988geometry} showed that  $\mathcal{T}_{g,n}$ admits a natural compactification,
whose boundary can be identified with $\mathcal{PMF}$.
A generic pair of transverse measured foliations $\mathcal{F},\mathcal{G}\in \mathcal{MF}$
determines a unique Teichm\"uller geodesic,
which has the projective classes of $\mathcal{F}$ and $\mathcal{G}$
as its ``limits"  on $\mathcal{PMF}$.

\subsection{Main theorems}
Level sets of extremal length functions in Teichm\"uller space, associated with measured foliations,
are called \emph{horospheres}.
The notion is motivated by the fact that extremal length functions on $\mathcal{T}_{g,n}$  are Hamiltonian functions of the Teichm\"uller horocycle flow
\cite{masur1995teichmuller}.

\begin{definition}
We say that a  diffeomorphism $f: \mathcal{T}_{g,n}\to \mathcal{T}_{g,n}$ \emph{preserves horospheres} if the image of any horosphere under $f$ is a horosphere.
\end{definition}

\begin{remark}
In this paper, we require that $f$ is a $C^1$-diffeomorphism.
The smoothness is just used to show that the inverse $f^{-1}$ also preserves horospheres (see Lemma \ref{Inverse_is_true}).
\end{remark}

 Our main result is the following:
\begin{theorem}\label{Result four}
Let $f: \mathcal{T}_{g,n}\to \mathcal{T}_{g,n}$ be a diffeomorphism that preserves horospheres.
Then $f$ is induced by an element of the extended mapping class group.
\end{theorem}

\begin{remark}
We exclude the case that $(g,n)=(1,0), (1,1)$ or $(0,4)$,
when $\T$ is isometric to the hyperbolic plane $\mathbb{H}^2$.
In this two-dimensional case, level sets of extremal length functions are horocycles in $\mathbb{H}^2$.
For instance, on the Teichm\"uller space of flat tori, any point $\tau\in \mathbb{H}^2$
 corresponds to a marked Riemann surface defined as
the quotient space of $\mathbb{C}$ by a lattice generated by $\langle z\mapsto z + 1, z\mapsto z + \tau\rangle$;
and the extremal length of the closed curve corresponding to $(1,0)$ is equal to $1/\mathrm{Im} \tau$.
It is not hard to check that if $f:\mathbb{H}^2\to \mathbb{H}^2$ is a diffeomorphism that preserves
horocycles, then $f$ also preserves geodesics. Any bijection between hyperbolic space that preserves geodesics
is an isometry \cite{JJ}. Thus $f\in \mathrm{PSL}(2,\mathbb{R})$. However, the mapping class group of
the torus is $\mathrm{PSL}(2,\mathbb{Z})$.
\end{remark}

The proof of Theorem \ref{Result four} is inspired by Ivanov's geometric proof of Royden's Theorem
 \cite{ivanov2001isometries}.
The vague idea is that, the action of $f$ on horospheres
should induce an action on the space of (projective) measured foliations.
In fact, there is a subset of $\mathcal{MF}$ with full measure,
on which the action induced by $f$ is an isomorphism.

Let us explain more details.
A measured foliation is \emph{indecomposable} if  it is equivalent
either to  a simple closed curve or to some minimal component with an ergodic measure
(see \S \ref{MF} for the precise definition).
Denote by $\mathcal{MF}_{ind}$ the set of indecomposable measured foliations.
It is well known that $\mathcal{MF}_{ind}$ is a subset of $\mathcal{MF}$ with full measure.
For $\mathcal{F} \in \mathcal{MF}$ and $X \in \mathcal{T}_{g,n}$,
we denote by $\operatorname{HS}(\mathcal{F}, X)$ the horosphere associated with $\F$
and passing through $X$.

With the above terminologies, we prove:
\begin{proposition}\label{Result two} Let $f: \mathcal{T}_{g,n}\to \mathcal{T}_{g,n}$ be a diffeomorphism that preserves horospheres.
  Assume that $\mathcal{F}\in \mathcal{MF}_{ind}$  and $f[\operatorname{HS}(\mathcal{F}, X)] = \operatorname{HS}(\mathcal{G}, Y)$. Then
  \begin{enumerate}
    \item $ f[\operatorname{HS}(\mathcal{F}, Z)] = \operatorname{HS}(\mathcal{G}, f(Z))$ for all $Z \in \mathcal{T}_{g,n}$;
    \item $\mathcal{G}\in \mathcal{MF}_{ind}$.
  \end{enumerate}
\end{proposition}
Thus $f$ induces a natural action on $\mathcal{MF}_{ind}$, denoted by $f_*$.
We further show that  $f_*$  preserves the relation of zero intersection
(see Proposition \ref{preserves indecomposable}).
There is a characterization of measured foliations
corresponding to simple closed curves in terms of the dimension of zero intersection subspace.
As a result, we can show that  $f_*$ induces an automorphism of the complex of curves of $S$.
It follows from a result of Ivanov that $f_*$
is given by an element of the extended mapping class group
(see Theorem \ref{Th B}).

To prove Theorem \ref{Result four}, we can reduce to the case that $f[\operatorname{HS}(\mathcal{F}, X)] = \operatorname{HS}(\mathcal{F},Y)$, for all $\mathcal{F}$ corresponding to simple closed curves.
We study the condition when two or three such horospheres are tangent to each other
(see Lemma \ref{condition of tangential} and Lemma \ref{Dense two}),
and use this to
show that $f$ is equal to the identify map on a dense subset of $\mathcal{T}_{g,n}$.
Then the continuity of $f$ implies that $f=\mathrm{id}$.

\bigskip

Using analytic nature of the Teichm\"uller metric, Royden \cite{royden1971automorphisms} (and extended by
Earle and Kra \cite{earle1974isometries}) proved that

\begin{theorem}[Royden]\label{thm:Royden}
If $3g-3+n\geq 2$, then every isometry of $\T$ with respect to the Teichm\"uller metric is induced by an element of
 the extended mapping class group.
\end{theorem}
Ivanov \cite{ivanov2001isometries} gave an alternative proof of Royden's Theorem,
by investigation of the asymptotic geometry of Teichm\"uller geodesic rays.
In \S \ref{sec:Busemann}, we observe that there is a direct relation between horospheres and level sets of Busemann functions, when the measured foliations defining the horospheres are indecomposable.
Consider any isometry $f$ of $\T$,
we show that $f$ preserves horospheres associated to indecomposable measured foliations.
Again, $f$ induces an isomorphism of $\mathcal{MF}_{ind}$. The proof of Theorem \ref{Result four}
can be adapted to show that $f$ is induced by an element of the extended mapping class group.
Thus we obtain a new proof of Royden's Theorem.

\subsection{Organization of the article}
In \S \ref{sec:preli} we give the preliminaries on Teichm\"uller theory and measured foliations.
The geometry of horospheres is investigated in \S \ref{sec:horo}.
We prove Proposition \ref{Result two} and Theorem \ref{Result four} in \S \ref{sec:diff}.
Theorem \ref{thm:Royden} is proved in \S \ref{sec:Royden}.

\bigskip

\emph{Acknowledgements.}
            The authors are grateful to Lixin Liu and Huiping Pan for their helpful suggestions and discussions. The authors would like to thank the referee for his (or her) corrections and
useful comments.

%W. Su is partially supported by NSFC
%No: 11671092, No: 11631010 and No: 11911530228.

\section{Preliminaries}\label{sec:preli}

In this section, we briefly recall the background material on Teichm\"uller theory of Riemann surfaces
and measured foliations.

\subsection{Teichm\"uller space}
Let $S$ be a Riemann surface of genus $g$ with $n$ punctures,  with $3g - 3 +n \geq 2$.
The \emph{Teichm\"uller space} $\mathcal{T}_{g,n}$ is the space of equivalence classes of pairs
$(X, f)$, where $f : S \rightarrow X$ is an orientation-preserving diffeomorphism (known as a marking).
The equivalence relation is given by $(X, f)\sim(Y, g)$ if there is a conformal mapping
$\phi : X \rightarrow Y$ so that $g^{-1} \circ \phi \circ f$ is isotopic to the identity map of $S$.

  The Teichm\"uller space $\mathcal{T}_{g,n}$ has
  a complete distance, called the \emph{Teichm\"uller distance}
  $d_{\mathcal{T}}(\cdot, \cdot)$. For any two points
  $[(X, f)], [(Y, g)] \in \mathcal{T}_{g,n}$ the distance is defined by
    $$ d_{\mathcal{T}}([(X, f)], [(Y, g)]) = \frac{1}{2} \inf_{h} \log K(h),$$
  where $h$ ranges over all quasiconformal mappings
  $ h: X \rightarrow Y $ such that $h \circ f$ is
  homotopic to $g$, and $K(h)$ is the maximal quasiconformal
  dilatation of $h$.

  For  simplicity, we shall denote a point in $\mathcal{T}_{g,n}$ by a Riemann surface $X$,
without explicit reference to the marking or to the equivalence relation.

\subsection{Measured foliations}\label{MF}
A \emph{measured foliation} $\mathcal{F}$ on $S$ is a foliation (with a finite number of singularities)
with a transverse invariant measure.
This means that if the local coordinates send the regular leaves of
$\mathcal{F}$ to horizontal arcs in $\mathbb{R}^{2}$,
then the transition functions on $\mathbb{R}^{2}$ are of the form $(f(x, y),  \pm y + c)$
where $c$ is a constant, and the measure is given by $|dy|$.
The allowed singularities of $\mathcal{F}$ are topologically the same as
those that occur at $z = 0$ in the line field of $z^{p-2}dz^{2}, p \geq 3 $
(we allow $p=1$ at the puncture of $S$).
  A leaf of $\mathcal{F}$ is called \emph{critical} if it contains
  a singularity of $\mathcal{F}$. The union of compact critical
  leaves is called the \emph{critical graph}.

  Let $\mathcal{S}$ be the set of free homotopy classes of non-trivial, non-peripheral simple closed curves on $S$. The \emph{intersection number} $i(\gamma, \mathcal{F})$ of a simple closed curve $\gamma$ with a measured foliation $\mathcal{F}$ endowed with transverse measure $\mu$ is defined by
    $$ i(\gamma, \mathcal{F}) = \inf_{\gamma'} \int_{\gamma'}d \mu,$$
  where the infimum is taken over all simple closed curves $\gamma'$ in the isotopy class of $\gamma$.

  Two measured foliations $\mathcal{F}$ and $\mathcal{F}'$ are  \emph{measure equivalent} if, for all $\gamma\in \mathcal{S}$, $i(\gamma, \mathcal{F}) = i(\gamma, \mathcal{F}')$.
  Denote by $ \mathcal{MF} = \mathcal{MF}(S)$ the space of equivalence classes of measured foliations on $S$.

 Two measured foliations $\mathcal{F}$ and $\mathcal{F}'$ are \emph{projectively equivalent} if there is a constant $b > 0$ such that $\mathcal{F}=b \cdot \mathcal{F'}$, i.e.  $i(\gamma, \mathcal{F}) =b \cdot i(\gamma, \mathcal{F}')$ for all $\gamma\in \mathcal{S}$. The space of projective equivalence classes of foliations is denoted by $\mathcal{PMF}$.

  Thurston showed that $\mathcal{MF}$ is homeomorphic to a $6g-6+2n$ dimensional ball and $\mathcal{PMF}$ is homeomorphic to a $6g-7+2n$ dimensional sphere. The set $\mathcal{S}$ is dense in $\PMF$. For more details on measured foliations, see \cite{fathi1979travaux}.

\bigskip

  We will use the ergodic decomposition of a measured foliation later in this paper.
  By removing the critical graph, a measured foliation $\mathcal{F}$ is decomposed into a finite number of
  connected components, each of which is either a cylinder foliated by  closed leaves
  or a minimal component on which every leaf is dense.
  Furthermore, the transverse measure on a minimal component $D$ can be represented  as
  a finite sum of  projectively distinct ergodic measures:
  $$\mu|_D =
  \sum_{k} \mu_{D,k}.$$
  We refer to \cite{Ivanov1992Subgroups, lenzhen2010criteria}
  for more details.

A measured foliation $\mathcal{F}'$ is called an
  \emph{indecomposable component} of $\mathcal{F}$ if it is either one of the
  cylindrical components of $\mathcal{F}$, or it is measure equivalent
  to some minimal component of $\mathcal{F}$ endowed with one of the ergodic measures.  A measured foliation $\mathcal{F}$ is \emph{indecomposable} if it has a unique
  indecomposable component. We denote the set of indecomposable measured foliations on $S$ by $\mathcal{MF}_{ind}$.

  Thus an indecomposable measured foliation is equivalent to
  either a weighted simple close curve or a minimal component on a subsurface with an ergodic measure.
  In particular, uniquely ergodic measured foliations are indecomposable. We recall that
  a measured foliation $\F$ is \emph{uniquely ergodic} if it is minimal and any
topologically equivalent measured foliation is measure equivalent to a multiple of $\F$.

Usually, we will represent a measured foliation $\mathcal{F}$ as a finite sum
    $$ \mathcal{F} = \sum_{i=1}^{k}\mathcal{F}_{i}$$
  of mutually disjoint ($i(\mathcal{F}_{i},\mathcal{F}_{j})=0$) and distinct indecomposable measured foliations.
  In the literature, such a (unique) decomposition is called the  \emph{ergodic decomposition} of
  $\mathcal{F}$ \cite{lenzhen2010criteria}.

  The next lemma will be used later.

\begin{lemma} \cite{Walsh2014The}\label{Walsh2014Lemma6.3}
  Let $\{\mathcal{F}_{i}\}_{i = 0}^{k}$ be a set of projectively distinct,
  indecomposable elements of $\mathcal{MF}$ such that
  $i(\mathcal{F}_{i}, \mathcal{F}_{j}) = 0$ for all $i$ and $j$.
  Then for any $\varepsilon > 0$,
  there exists a simple closed curve $\beta \in \mathcal{S}$ such that
  $$i(\mathcal{F}_{i}, \beta) < i(\mathcal{F}_{0}, \beta) \  \varepsilon, \ \forall \  i\neq 0.$$
\end{lemma}

\subsection{Quadratic differentials}
A holomorphic quadratic differential $q$ on $X \in \mathcal{T}_{g,n}$
  is a tensor which is locally represented by $q = q(z)dz^{2}$,
  where $q(z)$
  is a holomorphic function on the local conformal coordinate
  $z$ of $X$. We allow holomorphic quadratic differentials
  to have at most simple poles at the punctures of $X$. Denote
  the vector space of holomorphic quadratic differentials
  on $X$ by  $Q(X)$.

  The cotangent space
  of $\mathcal{T}_{g,n}$ at $X$ can be naturally identified with  $Q(X)$.
   We define the $L^{1}$-norm on $\mathcal{Q}(X)$ by
  $$||q|| = \int_{X}|q|.$$
  Denote by $\mathcal{QT}_{g,n}$ the cotangent bundle of
  $\mathcal{T}_{g,n}$, and let $\mathcal{Q}^{1}\mathcal{T}_{g,n}$ be
  the unit cotangent bundle of $\mathcal{T}_{g,n}$.

  A pair of  measured foliations $\{\mathcal{F}, \mathcal{G}\}$ is \emph{transverse} if
  $i(\F,\gamma)+i(\G,\gamma)>0$
  for all $\gamma\in \mathcal{MF}$.
Any $q \in Q(X)$ gives rise to a pair of transverse measured foliations
$\mathcal{F}_{v}(q)$ and $\mathcal{F}_{h}(q)$ on $X$, called
the \emph{vertical} and \emph{horizontal measured foliations} of $q$, respectively.
The vertical foliation $\mathcal{F}_{v}(q)$ (resp. horizontal foliation
$\mathcal{F}_{h}(q)$) is defined by the foliation of the direction field $q(z)d z^2<0$
(resp. $q(z)d z^2>0$) with the transverse measure $|\textrm{Re}\sqrt{q}|$ (resp.$|\textrm{Im}\sqrt{q}|$).

On the other hand, according to a fundamental result of Hubbard and Masur \cite{hubbard1979quadratic},
for any measured foliation $\mathcal{F}\in \mathcal{MF}$, there is a unique holomorphic quadratic differential
$q \in Q(X)$ such that $\mathcal{F}_{v}(q)$ is measure equivalent to $\mathcal{F}$.
The quadratic differential $q$ is called the \emph{Hubbard-Masur differential} of $\mathcal{F}$.

Let $X=(X,f)\in \mathcal{T}_{g,n}$ and $q \in Q(X)$. For any $t \in \mathbb{R}$,
 consider the normalized solution $f_{t}$ of the Beltrami equation
    \begin{equation*}
      \frac{\partial f}{\partial \overline{z}} = \tanh(t) \frac{|q|}{q} \frac{\partial f}{\partial z}
    \end{equation*}
  on $X$. We obtain a geodesic in the Teichm\"uller space:
    \begin{eqnarray*}
      \mathbf{G}_{q}: \mathbb{R} &\to&  \mathcal{T}_{g,n}, \\
      t &\mapsto&  (X_{t}, f_{t} \circ f).
    \end{eqnarray*}
  where $X_{t} = f_{t}(X)$. We call $\textbf{G}_{q}$ the \emph{Teichm\"uller geodesic associated to $q$}.

  We say a Teichm\"uller geodesic is determined by a pair of transverse measured foliations $\{\mathcal{F}, \mathcal{G}\}$, if it is defined by a holomorphic quadratic differential whose vertical and horizontal foliations are in the projective  classes of $\mathcal{F}$ and $\mathcal{G}$. Any pair of transverse measured foliations determines a unique
  Teichm\"uller geodesic \cite{gardiner1991extremal}.

\subsection{Extremal length}
Extremal length is an important tool in the study of the Teichm\"uller metric.
The notion is due to Ahlfors and Beurling.

  Let $X=(X,f) \in \mathcal{T}_{g,n}$ and $\alpha \in \mathcal{S}$.
  The extremal length $\operatorname{Ext}_{X}(\alpha)$ is defined by
    $$ \operatorname{Ext}_{X}(\alpha) = \sup_{\rho}
       \frac{\ell_{\rho}(f(\alpha))^{2}}{\textrm{Area}(X, \rho)},$$
  where the supremum is taken over all conformal metrics $\rho$
  on $X$ and $\ell_{\rho}(f(\alpha))$
  denotes the geodesic length of $f(\alpha)$ in the metric $\rho$.
  Kerckhoff \cite{kerckhoff1980asymptotic} proved
  that the definition of extremal length extends continuously to $\mathcal{MF}$.
  One can show that
  the extremal length of a measured foliation $\mathcal{F}$ satisfies
  $$\operatorname{Ext}_{X}(\mathcal{F})=\|q\|,$$
  where $q$ is the Hubbard-Masur differential of $\mathcal{F}$.

    The following formula of Kerckhoff \cite{kerckhoff1980asymptotic} is very useful to understand the geometry of Teichm\"uller distance.
\begin{theorem}
  For any $X, Y \in \mathcal{T}_{g,n}$, the Teichm\"uller distance
  between $X$ and $Y$ is given by
    $$ d_{\mathcal{T}}(X, Y) = \frac{1}{2} \log \sup_{\alpha \in
       \mathcal{S}}\frac{\operatorname{Ext}_{X}(\alpha)}
       {\operatorname{Ext}_{Y}(\alpha)}.$$
\end{theorem}

The following inequality is due to Minsky \cite{minsky1993teichmuller},
see also Gardiner-Masur \cite{gardiner1991extremal}.

\begin{theorem}[Minsky]\label{Minsky}
  Let $\{\mathcal{F}, \mathcal{G}\}\in \mathcal{MF}$ be a pair of transverse measured foliations.
  Then for any $X\in \mathcal{T}_{g,n}$, we have
    $$ i(\mathcal{F}, \mathcal{G})^{2} \leq \operatorname{Ext}_{X}(\mathcal{F})\operatorname{Ext}_{X}(\mathcal{G}).$$
  Moreover, the equality is obtained if and only if $X$ belongs to the
 unique Teichm\"uller geodesic determined by $\mathcal{F}$ and $\mathcal{G}$ (i.e., the horizontal and vertical foliations are in the projective  classes of $\mathcal{F}$ and $\mathcal{G}$).
\end{theorem}

\begin{corollary}\label{coro:ext}
For any $X\in\T$ and $\F\in\MF$, we have
  $$\operatorname{Ext}_{X}(\F)=\sup_{\gamma\in\mathcal{S}} \frac{i(\F, \gamma)^2}{\operatorname{Ext}_{X}(\gamma)}.$$
  \end{corollary}
  \begin{proof}
By Minsky's inequality,
$$\operatorname{Ext}_{X}(\F)\geq \sup_{\gamma\in\mathcal{S}} \frac{i(\F, \gamma)^2}{\operatorname{Ext}_{X}(\gamma)}.$$
On the other hand, let $q$ be the Hubbard-Masur differential of $\F$, and let $\G$ be the horizontal measured foliation of
$q$. Then
$$\operatorname{Ext}_{X}(\F)= \frac{i(\F, \G)^2}{\operatorname{Ext}_{X}(\G)}.$$
By the density of weighted simple closed curves in $\MF$, we are done.
\end{proof}

The following first variational formula of extremal length is called Gardiner's formula.

\begin{theorem}\cite{Gardiner,Gardiner2}\label{thm:Gardiner}
Let $\mu=\mu(z)\frac{d \bar z}{dz}$ be a Beltrami differential on $X$ that represents a tangent vector of
$\mathcal{T}(S)$ at $X$. Then for any measured foliation $\F\in \MF$,
$$ d \operatorname{Ext}_{X}(\mathcal{F})[\mu] = 2 \operatorname{Re} \iint_{X} \mu(z) q(z) \  dxdy,$$
 where $q$ is the Hubbard-Masur differential of $\mathcal{F}$.
\end{theorem}

\subsection{Complex of curves and mapping class group}\label{CCaMCG}
  The \textit{complex of curves}  was introduced into the study of Teichm\"uller spaces by Harvey \cite{harvey1981boundary},
   as an analogue of the Tits building of a symmetric space.
  The vertex set of the complex of curves $\mathcal{C}(S)$ is given by
   $\mathcal{S}$.
   Two vertices $\alpha, \beta \in \mathcal{S}$ are connected by an edge if they have disjoint representations.
   For any two vertices $\alpha, \beta$,
    we define the distance $d_{\mathcal{S}}(\alpha, \beta)$ to be the minimal number of edges connecting $\alpha$ and $\beta$.

  The \emph{mapping class group} $\mathrm{Mod}(S)$ is the group of  homotopy classes
of orientation-preserving diffeomorphisms  $\sigma: S\to S$. Every mapping class
$[\sigma]$ acts on $\mathcal{T}_{g,n}$ by changing the markings:
$$[(X,f)] \to [(X,f\circ \sigma^{-1})].$$
Denote by $\mathrm{Mod}^{\pm}(S)$ the extended mapping class group, which contains $\operatorname{Mod}(S)$ as a subgroup of index two.

 It is clear that $\mathrm{Mod}^{\pm}(S)$ acts on $\mathcal{C}(S)$ as a group of automorphisms.

\begin{theorem}\label{Th B}
  If $S$ is not a sphere with $\leq 4$ punctures, nor a torus with
$\leq 2$ punctures, then every automorphism of $\mathcal{C}(S)$ is given by an element of
$\mathrm{Mod}^{\pm}(S)$.
\end{theorem}
The above theorem is proved by Ivanov \cite{ivanov1997automorphisms} for surfaces of genus $g\geq 2$ and by Korkmaz \cite{korkmaz1997complexes} for
$g\geq 1$. See also Luo \cite{Luo}. We remark that,
for the torus with two punctures, $\mathcal{C}(S_{1,2})$ is isomorphic to $\mathcal{C}(S_{0,5})$. Instead of $\mathrm{Mod}^{\pm}(S_{1,2})$,
the automorphism group of  $\mathcal{C}(S_{1,2})$
is $\mathrm{Mod}^{\pm}(S_{0,5})$,

\section{The geometry of horospheres}\label{sec:horo}
In this section, we study the geometry of horospheres in $\mathcal{T}_{g,n}$.
Although the Teichm\"uller metric is neither non-positively curved nor $\delta$-hyperbolic,
we will show that the asymptotic geometry of horospheres,
for those associated with indecomposable measured foliations,
have similar properties to those in the hyperbolic
space.

\subsection{Horoballs and horospheres}\label{HaH}
 Given $\mathcal{F} \in \mathcal{MF}$, the extremal length function $X \mapsto \operatorname{Ext}_{X}(\mathcal{F})$
   is a $C^1$-function on $\mathcal{T}_{g,n}$.

\begin{definition}\label{def:3.1}
Let $\mathcal{F} \in \mathcal{MF}$ and $s\in \mathbb{R}_+$.
The open horoball associated to $\mathcal{F}$ is defined by
  $$\operatorname{HB}(\mathcal{F}, s) = \{X \in \mathcal{T}_{g,n}\ |\ \operatorname{Ext}_{X}(\mathcal{F}) < s\}.$$
The associated closed horoball is defined as
$$\overline{\operatorname{HB}}(\mathcal{F}, s) = \{X \in \mathcal{T}_{g,n}\ |\ \operatorname{Ext}_{X}(\mathcal{F}) \leq s\}.$$
 The associated  horosphere is defined as
  $$\operatorname{HS}(\mathcal{F}, s) = \{ X \in \mathcal{T}_{g,n}\ |\ \operatorname{Ext}_{X}(\mathcal{F}) = s\}.$$
\end{definition}

\begin{remark}\label{coincide}
  Let $\mathcal{F} \in \mathcal{MF}$ and $k, s \in \mathbb{R}_{+}$.  Since $ \operatorname{Ext}_{X}(k\cdot\mathcal{F}) = k^{2}\operatorname{Ext}_{X}(\mathcal{F}),$
  we have
    $$\operatorname{HS}(k\cdot\mathcal{F}, k^{2}s) = \operatorname{HS}(\mathcal{F}, s).$$

\end{remark}

\begin{lemma}
  For any $X, Y \in \mathcal{T}_{g,n}$, there exists a measured foliation $\mathcal{F} \in \mathcal{MF}$ and $t \in \mathbb{R}_{+}$ such that $X, Y \in \operatorname{HS}(\mathcal{F}, t)$.
\end{lemma}
\begin{proof}
  It suffices to prove that there is a measured foliation $\mathcal{F} \in \mathcal{MF}$ such that $\operatorname{Ext}_{X}(\mathcal{F}) = \operatorname{Ext}_{Y}(\mathcal{F})$.
  Let $\Phi : \mathcal{MF} \rightarrow \mathbb{R}$ defined by
  $$ \Phi(\mathcal{F}) = \operatorname{Ext}_{X}(\mathcal{F}) - \operatorname{Ext}_{Y}(\mathcal{F}).$$
  It is obvious that $\Phi$ is continuous.
  Let $f : X \rightarrow Y$ be the Teichm\"uller map from $X$ to $Y$, and $q \in Q(X)$ be the holomorphic quadratic differential associated to $f$. Denote by $\mathcal{F}_{h}(q)$ (resp. $\mathcal{F}_{v}(q))$ the horizonal foliation (resp. vertical foliation) of $q$. Then we have $ \Phi(\mathcal{F}_{h}(q))<0$ and $ \Phi(\mathcal{F}_{v}(q))>0$.
  Since $\mathcal{MF}\setminus \{0\}$ is connected, the mean value theorem of continuous function implies that there is a measured foliation $\mathcal{F}^{'}$ such that $\Phi(\mathcal{F}^{'}) = 0$.
\end{proof}

\begin{remark}
  In general, the horosphere containing $X, Y \in \mathcal{T}_{g,n}$ is not unique.
\end{remark}

\bigskip

  Let $\mathcal{F} \in \mathcal{MF}$ and $X \in \mathcal{T}_{g,n}$, we denote by
    $$ \operatorname{HS}(\mathcal{F}, X) = \{Y \in \mathcal{T}_{g,n}\ |\ \operatorname{Ext}_{Y}(\mathcal{F}) = \operatorname{Ext}_{X}(\mathcal{F})\}.$$
Let $T_{X} = T_{X}\mathcal{T}_{g,n}$ denote the tangent space of $\mathcal{T}_{g,n}$ at $X$,
   and let $\mathrm{Gr}(T_{X})$ denote the set of linear subspaces of $T_{X}$ of dimension $6g - 7 + 2n$.

  With the above notation, we define a map $ T : \mathcal{MF} \rightarrow \mathrm{Gr}(T_{X})$
  by
    $$ T(\mathcal{F}) = T_{X}\operatorname{HS}(\mathcal{F}, X),$$
  where $T_{X}\operatorname{HS}(\mathcal{F}, X)$ denotes the tangent space of $\operatorname{HS}(\mathcal{F}, X)$ at $X$.

\begin{lemma}\label{necessary condition of tangent}
  Let $\mathcal{F}, \mathcal{G} \in \mathcal{MF}$. Then $T(\mathcal{F}) = T(\mathcal{G})$ if and only if $\mathcal{G}$ is projectively equivalent to $\mathcal{F}_{h}(q)$ or $\mathcal{F}_{v}(q)$, where $q \in Q(X)$ is the Hubbard-Masur differential of  $\mathcal{F}$.
\end{lemma}
\begin{proof} Assume that $T(\mathcal{F}) = T(\mathcal{G})$.
  According to the definition of tangent space and the $C^{1}$-property of extremal length function $\operatorname{Ext}_{X}(\mathcal{F})$, we have
  \begin{equation*}
    d \operatorname{Ext}_{X}(\mathcal{F})[\mu] = d \operatorname{Ext}_{X}(\mathcal{G})[\mu] = 0, \  \forall \  \mu \in T_{X}\operatorname{HS}(\mathcal{F},X).
  \end{equation*}
   Moreover, since the subspace of $T_X$ tangent to the horosphere has codimension one,
    there is a non-zero constant  $k \in \mathbb{R}$ such that
  \begin{equation*}
    d \operatorname{Ext}_{X}(\mathcal{F})[\nu] = k \cdot  d \operatorname{Ext}_{X}(\mathcal{G})[\nu] \neq 0 , \  \forall \ \nu \in T_{X} \setminus T_{X}\operatorname{HS}(\mathcal{F}, X).
  \end{equation*}
It follows that $d \operatorname{Ext}_{X}(\mathcal{F})$ is a real multiple of $d \operatorname{Ext}_{X}(\mathcal{G})$.

Let $q_{1}$ and $q_{2}$ be the holomorphic quadratic differentials on $X$ which realize the measured foliations $\mathcal{F}$ and $\mathcal{G}$, respectively. Using Gardiner's formula (Theorem \ref{thm:Gardiner}), we have
 \begin{eqnarray*}
 d \operatorname{Ext}_{X}(\mathcal{F})[\mu] &=& 2 \operatorname{Re} \iint_{X} \mu(z) q_1(z)dxdy \\
 &=& k \cdot \left( 2 \operatorname{Re} \iint_{X} \mu(z) q_2(z)dxdy\right),  \ \forall \  \mu \in T_X.
 \end{eqnarray*}
 This implies that $q_{1} = k q_{2}$. As a result, if $k > 0$, then $\mathcal{G}$ is projectively equivalent to $\mathcal{F}$; if $k < 0$, then $\mathcal{G}$ is projectively equivalent to $\mathcal{F}_{h}(q_{1})$.

 To prove  $T(\mathcal{F}) = T(\mathcal{G})$ under the assumption that $\mathcal{G}$ is projectively equivalent to $\mathcal{F}_{h}(q)$ or $\mathcal{F}_{v}(q)$, we can apply the above argument in the converse direction.
\end{proof}

\begin{corollary}\label{coro:3.5}
   If $\mathcal{G}$ is not projectively equivalent to $\mathcal{F}$, then $\operatorname{HS}(\mathcal{G}, X) \neq \operatorname{HS}(\mathcal{F}, X)$.
\end{corollary}

  Let $V \in \mathrm{Gr}(T_{X})$. Suppose that $\langle \mu_{1}, \cdots, \mu_{6g-7 + 2n}\rangle=V$. Each $\mu_{i}$ induces a linear function
    $$\hat{\mu}_{i} : Q(X) \rightarrow \mathbb{R} $$
  by
    $$ \hat{\mu}_{i}(q)= \langle \mu_{i}, q\rangle = 2 \operatorname{Re} \int_{X}\mu_{i}(z)q(z) dxdy.$$
  Let $V_{i} = \operatorname{Ker} (\hat{\mu}_{i})$.
  It is clear that $V_{i}$ is a linear subspace of $Q(X)$ and $\operatorname{dim}(V_{i}) = 6g - 7 + 2n$. Let
    $$ V^{*} = \cap_{i = 1}^{6g - 7 + 2n} V_{i}.$$
It follows from linear algebra that  $V^{*}$  is a linear subspace with $\operatorname{dim}(V^{*}) \geq 1$. This implies that there is a $q \in V^{*}$ and $q \neq 0$ such that
    $$ \langle\mu_{i}, q\rangle = 0, i = 1, \cdots, 6g - 7 + 2n.$$
  Hence $T(\mathcal{F}_{h}(q)) = \langle \mu_{1}, \cdots, \mu_{6g- 7 + 2n}\rangle$. This shows that
   \begin{lemma}\label{lemma:gr}
   The map $T$ defined above is surjective. Thus for any linear subspace $V \in \mathrm{Gr}(T_{X})$,
   there is a measured foliation $\mathcal{F}$ such that the tangent space of horosphere $\operatorname{HS}(\mathcal{F}, X)$ at $X$ is $V$.
\end{lemma}

Gardiner and Masur \cite{gardiner1991extremal} proved that
\begin{lemma}
  Every horosphere in $\mathcal{T}_{g,n}$ is a hypersurface homeomorphic to the Euclidean space $\mathbb{R}^{6g - 7 + 2n}$.
\end{lemma}

\subsection{Relation between horospheres}

  Let $X \in \mathcal{T}_{g,n}$ and let $\mathbf{A}$ be a subset of $\mathcal{T}_{g,n}$, we define
    $$ d_{\mathcal{T}}(X, \mathbf{A}) = \inf_{Y \in \mathbf{A}} d_{\mathcal{T}}(X, Y).$$
  If there is a point $Y \in \mathbf{A}$ such that $d_{\mathcal{T}}(X, Y) = d_{\mathcal{T}}(X, \mathbf{A})$, then $Y$ is called a \textit{foot} of $X$ on $\mathbf{A}$.

\begin{lemma}\label{Equidistant}
  Let $0 < s < t$ and $\mathcal{F} \in \mathcal{MF}$. Then horospheres $\operatorname{HS}(\mathcal{F}, s)$ and $\operatorname{HS}(\mathcal{F}, t)$ are equidistant, i.e. for any $X \in \operatorname{HS}(\mathcal{F}, s)$, we have
  $$ d_{\mathcal{T}}(X, \operatorname{HS}(\mathcal{F}, t)) = \frac{1}{2}\log \frac{t}{s}.$$
  Moreover, any $X \in \operatorname{HS}(\mathcal{F}, s)$ has a unique foot on $\operatorname{HS}(\mathcal{F}, t)$.
\end{lemma}
\begin{proof}
According to Kerckhoff's formula, we have
    $$ d_{\mathcal{T}}(X, Y) \geq \frac{1}{2} \log \frac{\operatorname{Ext}_{Y}(\mathcal{F})}{\operatorname{Ext}_{X}(\mathcal{F})}
    = \frac{1}{2} \log \frac{t}{s}$$
    for any $X \in \operatorname{HS}(\mathcal{F}, s), Y\in \operatorname{HS}(\mathcal{F}, t).$
    Thus $$ d_{\mathcal{T}}(X, \operatorname{HS}(\mathcal{F}, t))
    \geq \frac{1}{2}\log \frac{t}{s}.$$

  If we choose $Y\in \operatorname{HS}(\mathcal{F}, t)$ as the Teichm\"uller deformation of $X$ in the direction $q\in Q(X)$ with $\mathcal{F}_v(q)=\mathcal{F}$,
  then $$ d_{\mathcal{T}}(X,Y)= {\frac{1}{2}}\log \frac{t}{s}.$$
  As a result, $$ d_{\mathcal{T}}(X, \operatorname{HS}(\mathcal{F}, t))= \frac{1}{2} \log \frac{t}{s}.$$

  Note that $Y\in \operatorname{HS}(\mathcal{F}, t)$ is a foot of $X$ if and only if
 $$  d_{\mathcal{T}}(X, Y)= \frac{1}{2} \log \frac{\operatorname{Ext}_{Y}(\mathcal{F})}{\operatorname{Ext}_{X}(\mathcal{F})}.$$
   By the uniqueness of Teichm\"uller map, the above equality holds if and only if $Y$ is the Teichm\"uller deformation
   of $X$ in the direction $q\in Q(X)$ with $\mathcal{F}_v(q)=\mathcal{F}$. This implies that the foot $Y$ is unique.
\end{proof}
We call the Teichm\"uller geodesic passing through $X\in \operatorname{HS}(\mathcal{F}, s)$ and $Y\in \operatorname{HS}(\mathcal{F}, t)$ such that $Y$ is the foot of $X$ on $\operatorname{HS}(\mathcal{F}, t)$
 a geodesic \emph{perpendicular} to the family of horospheres $\operatorname{HS}(\mathcal{F}, s), s\in \mathbb{R}_+$.

\bigskip

To obtain further results, we first consider the asymptotic estimates of extremal length functions
on a given horosphere.

Fix a horosphere\ $\operatorname{HS}(\mathcal{F}, t)$. Consider the function $\operatorname{Ext}_{X}(\mathcal{G})$ for any $\mathcal{G} \in \mathcal{MF}$,
where $X$ runs over all points belong to $\operatorname{HS}(\mathcal{F}, t)$.
We will write $\mathcal{F}=\sum_i \mathcal{F}_i$ as the ergodic decomposition of $\mathcal{F}$.
If each indecomposable component of $\mathcal{G}$ is projectively equivalent to one
of the indecomposable components of $\mathcal{F}$,
we denote by $\mathcal{G}\prec \mathcal{F}$;
otherwise, $\mathcal{G}\nprec \mathcal{F}$.

\begin{lemma}\label{Estimate of extremal length}
    Fix a horosphere $\operatorname{HS}(\mathcal{F}, t)$. For any $\mathcal{G} \in \mathcal{MF}$, we have
    \begin{enumerate}
      \item If $i(\mathcal{F}, \mathcal{G}) \neq 0$, then
            $$\inf_{X\in \operatorname{HS}(\mathcal{F}, t)} \operatorname{Ext}_{X}(\mathcal{G})>0,
            \sup_{X\in \operatorname{HS}(\mathcal{F}, t)} \operatorname{Ext}_{X}(\mathcal{G})=\infty.$$
      \item If $i(\mathcal{F}, \mathcal{G}) = 0$ and $\mathcal{G}\prec \mathcal{F}$, then $$\sup_{X\in \operatorname{HS}(\mathcal{F}, t)} \operatorname{Ext}_{X}(\mathcal{G})<\infty.$$
      \item If $i(\mathcal{F}, \mathcal{G}) = 0$ and $\mathcal{G}\nprec \mathcal{F}$,
      then $$\sup_{X\in \operatorname{HS}(\mathcal{F}, t)} \operatorname{Ext}_{X}(\mathcal{G})=\infty.$$
    \end{enumerate}
\end{lemma}
\begin{proof}
(1) According to Minsky's inequality (Theorem \ref{Minsky}), we have
  $$ \operatorname{Ext}_{X}(\mathcal{F}) \operatorname{Ext}_{X}(\mathcal{G}) \geq i(\mathcal{F}, \mathcal{G})^2 > 0.$$
Then
 $$ \inf_{X\in \operatorname{HS}(\mathcal{F}, t)} \operatorname{Ext}_{X}(\mathcal{G})\geq \frac{i(\mathcal{F}, \mathcal{G})^2}
 {t}.$$

For the supremum, we use the action of horocycle flow on $\T$.  Choose any $X\in \operatorname{HS}(\mathcal{F}, t)$.
Denote by $q$ the Hubbard-Masur differential of $\mathcal{F}$ on $X$ and $\mathcal{G}'$ the horizontal measured foliations of $q$. In local coordinates $z=x+iy$  on which $q = d z^2$, the horocycle flow $h^s: \Q \to \Q$ acts on $q$
by

\begin{equation*}
\left(
  \begin{array}{c}
    d y \\
    d x \\
  \end{array}
\right)
\mapsto
\left(
  \begin{array}{cc}
    1 &  s\\
    0 &  1\\
  \end{array}
\right)
\left(
  \begin{array}{c}
    d y \\
   d x \\
  \end{array}
\right)
=
\left(
  \begin{array}{c}
    d y+ s d x \\
    d x \\
  \end{array}
\right).
\end{equation*}
The horocycle flow acts on the horosphere $\operatorname{HS}(\mathcal{F}, t)$. For any closed curve $\gamma \in \mathcal{S}$,
its length under the flat metric $|h^s(q)|$ has an explicit lower bound:
\begin{eqnarray*}
\int_\gamma |s d x +  d y| &\geq& |s|i(\F,\gamma)-i(\G'\gamma).
\end{eqnarray*}
The area of $h^s(q)$ is equal to $\|q\|$.
Denote by $X_s$ the projection of $h^s(q)$ on $\mathcal{T}_{g,n}$.
Then $X_{s} \in \textrm{HS}(\F, t)$.
By definition of extremal length, we have
\begin{eqnarray*}
\operatorname{Ext}_{X_s}(\gamma)&\geq& \frac{s^{2}i(\F,\gamma)^2-2s i(\F,\gamma)i(\G',\gamma)+i(\G',\gamma)^{2}}{\|q\|}\\
&\geq& \frac{s^{2}i(\F,\gamma)^2}{2\|q\|}
\end{eqnarray*}
when $s$ is sufficiently large.
 By continuity, the above inequality applies to general measured foliations. In particular, when $i(\F, \G) \neq 0$, we have
 $$ \sup_{X\in \operatorname{HS}(\mathcal{F}, t)} \operatorname{Ext}_{X}(\mathcal{G})=\infty.$$

(2) Decompose $\mathcal{F}$ into its indecomposable components
    $$ \mathcal{F} = \sum_{i=1}^{k}\mathcal{F}_{i}.$$
  According to Lenzhen-Masur \cite[Theorem C]{lenzhen2010criteria}, there exists a sequence of multi-curves
  $$\sum_{i=1}^{k} s^{i}_{n}\gamma^{i}_{n} \rightarrow \mathcal{F},\   s^{i}_{n}\gamma^{i}_{n} \rightarrow  \mathcal{F}_{i},\  n \rightarrow \infty.$$
  Note that each $\gamma^{i}_{n}$ may itself be a multi-curve. By continuity, we have
    $\operatorname{Ext}_{X}(\sum_{i=1}^{k} s^{i}_{n}\gamma^{i}_{n}) \rightarrow \operatorname{Ext}_{X}(\mathcal{F})$
  as
    $n \rightarrow \infty$.
  By definition of extremal length, we have
    $$ \operatorname{Ext}_{X}(s^{i}_{n}\gamma^{i}_{n}) \leq \operatorname{Ext}_{X}(\sum_{j=1}^{k} s^{j}_{n}\gamma^{j}_{n}).$$
  Then we have
    $$ \operatorname{Ext}_{X}(\mathcal{F}_{i}) \leq \operatorname{Ext}_{X}(\mathcal{F})=t.$$

  Since $\mathcal{G}\prec \mathcal{F}$, we can write $\mathcal{G}$ as
  $\mathcal{G}  = \sum_{i=1}^k a_{i}\mathcal{F}_{i}, \ a_i\geq 0 $.
  Use the definition of extremal length and the above result of Lenzhen-Masur again, we have
 \begin{eqnarray*}
 \operatorname{Ext}_{X}(\sum_{i=1}^k a_{i}\mathcal{F}_{i})
  &\leq& \left( \sum_{i=1}^k a_{i}\sqrt{\operatorname{Ext}_{X}(\mathcal{F}_{i})} \right)^2 \\
 &\leq& \left( \sum_{i=1}^k a_{i} \right)^2 t.
 \end{eqnarray*}

(3) Without loss of generality, we may assume that $\mathcal{G}$ is indecomposable and $\mathcal{G}$ is disjoint from  $\mathcal{F}$. By Lemma \ref{Walsh2014Lemma6.3}, for any $M \in \mathbb{R}_{+}$ there is $\beta \in \mathcal{S}$ such that
  \begin{equation*}
    \frac{i(\mathcal{G}, \beta)}{i(\mathcal{F}, \beta)} \geq M.
  \end{equation*}
  Since uniquely ergodic measured foliations are dense in $\mathcal{MF}$,
  there is a uniquely ergodic measured foliation $\mathcal{F}'$ such that
  \begin{equation*}
    \frac{i(\mathcal{G}, \mathcal{F}')}{i(\mathcal{F}, \mathcal{F}')} \geq \frac{M}{2}.
  \end{equation*}
  It is clear that $\{\mathcal{F}, \mathcal{F}'\}$ is a pair of transverse measured foliations.
  By Theorem \ref{Minsky},  there is a point $X \in \operatorname{HS}(\mathcal{F}, t)$ satisfying
  \begin{equation*}
    \operatorname{Ext}_{X}(\mathcal{F}) = \frac{i(\mathcal{F}, \mathcal{F}')^{2}}{\operatorname{Ext}_{X}(\mathcal{F}')}.
  \end{equation*}
  In fact, $X$ is the intersection point of $\operatorname{HS}(\mathcal{F}, t)$ with the Teichm\"uller geodesic
  determined by $\mathcal{F}$ and $\mathcal{F}'$.
It follows that
  \begin{equation*}
    \operatorname{Ext}_{X}(\mathcal{G}) \geq \frac{i(\mathcal{G}, \mathcal{F}')^{2}}{\operatorname{Ext}_{X}(\mathcal{F}')} \geq (\frac{M}{2})^2 \operatorname{Ext}_{X}(\mathcal{F}).
  \end{equation*}
  This implies that $$\sup_{X \in \operatorname{HS}(\mathcal{F}, t)}\operatorname{Ext}_{X}(\mathcal{G}) = + \infty.$$
The proof is complete.
 \end{proof}

\begin{remark}\label{rem:inclusion}
  If $\operatorname{HS}(\mathcal{F}, t)$ is a horosphere and $\mathcal{G} \in \mathcal{MF}$ satisfies
    $$ \sup_{X \in \operatorname{HS}(\mathcal{F}, t)}\operatorname{Ext}_{X}(\mathcal{G}) \leq M$$
  for some $M \in \mathbb{R}_{+}$, then
    $$ \operatorname{HB}(\mathcal{F}, t) \subset \operatorname{HB}(\mathcal{G}, s) $$
  when $s \geq M$. We recall that $\operatorname{HB}(\mathcal{F}, t)$ is the horoball whose boundary is
  $\operatorname{HS}(\mathcal{F}, t)$, see Definition \ref{def:3.1}.
\end{remark}

 The following corollary is immediate:

\begin{corollary}\label{Relations}
  Given any horosphere $\operatorname{HS}(\mathcal{F}, t)$, we have
  \begin{enumerate}
    \item If $i(\mathcal{F}, \mathcal{G}) \neq 0$, then there exists $s_{0} = s_0(\mathcal{G}, t) > 0$ such that
        $$ \operatorname{HB}(\mathcal{F}, t) \cap \operatorname{HB}(\mathcal{G}, s) = \emptyset, \ s < s_{0}.$$
    \item If $i(\mathcal{F}, \mathcal{G}) = 0$ and  $\mathcal{G}\prec \mathcal{F}$,
    then there exists $s_{0} = s_{0}(\mathcal{G}, t) > 0$ such that
            $$ \operatorname{HB}(\mathcal{F}, t) \subset \operatorname{HB}(\mathcal{G}, s), \ s > s_{0}.$$
  \end{enumerate}
\end{corollary}

%\begin{corollary}\label{parallel}
%  Let $\mathcal{F}, G \in \mathcal{MF}$ . If $\mathcal{F}$ and $G$ have the same indecomposable components. Then for any $s \in \mathbb{R}_{+}$, there exist $t_{1}, t_{2} \in \mathbb{R}_{+}$ and $t_{2} > t_{1}$ such that
%  $$ \operatorname{HS}(\mathcal{F}, s) \subset \operatorname{HB}(G, t_{2})/\operatorname{HB}(G, t_{1}).$$
%\end{corollary}
%
%This means that $\operatorname{HS}(\mathcal{F}, s)$ is between $\operatorname{HS}(G, t_{1})$ and $\operatorname{HS}(G, t_{2})$.

Recall that $\mathcal{MF}_{ind}$ is defined as the set of indecomposable measured foliations on $S$.

\begin{proposition}\label{Main-Lemma-One}
  Let $\mathcal{F} \in \mathcal{MF}_{ind}$ and $s \in \mathbb{R}_{+}$. If there exist $\mathcal{G} \in \mathcal{MF}$ and $t \in \mathbb{R}_{+}$ such that
  \begin{equation*}
    \operatorname{HB}(\mathcal{F}, s) \subset \operatorname{HB}(\mathcal{G}, t),
  \end{equation*}
  then $\mathcal{G} = k \mathcal{F}$, for some $k \in \mathbb{R}_{+}$.
\end{proposition}
\begin{proof}
  By definition,
  $$ \sup_{X \in \operatorname{HB}(\mathcal{F}, s)} \operatorname{Ext}_{X}(\mathcal{G}) \leq t.$$
  It follows from Lemma \ref{Estimate of extremal length} that $\G\prec \F$. Since  $\mathcal{F}$ is indecomposable,
  $\mathcal{G}$ must be a multiple of $\mathcal{F}$.
\end{proof}

\subsection{Horospheres tangent to each other}

\begin{definition}\label{tangential}
Two horospheres $\operatorname{HS}(\mathcal{F}, s)$ and $\operatorname{HS}(\mathcal{G}, t)$ are tangent to each other if they satisfy the following conditions:
  \begin{enumerate}
    \item $\operatorname{HB}(\mathcal{F}, s) \cap \operatorname{HB}(\mathcal{G}, t) = \emptyset$,
    \item $\operatorname{HS}(\mathcal{F}, s) \cap \operatorname{HS}(\mathcal{G}, t) \neq \emptyset$.
  \end{enumerate}
\end{definition}

%\begin{figure}[htbp]
%\centering
%\includegraphics[width=6cm]{tangent.pdf}
%\end{figure}

If $\operatorname{HS}(\mathcal{F}, s)$ and $\operatorname{HS}(\mathcal{G}, t)$ are tangent to each other,
then it is necessary that $\{\mathcal{F}, \G\}$ is a pair of transverse measured foliations.
 This follows from Lemma \ref{necessary condition of tangent}. In fact, at any $X\in \operatorname{HS}(\mathcal{F}, s)\cap\operatorname{HS}(\mathcal{G}, t)$, we have
 $T(\F)=T(\G)$ (otherwise, one can check that the  hypothesis $\operatorname{HB}(\mathcal{F}, s) \cap \operatorname{HB}(\mathcal{G}, t) = \emptyset$ is not satisfied).
And then $\F$ and $\G$ should be equivalent to the vertical and horizontal measured foliations of some
quadratic differential.

The following lemma will show that, if two horospheres $\operatorname{HS}(\mathcal{F}, s)$ and $\operatorname{HS}(\mathcal{G}, t)$ are tangent to each other,
then they have a unique intersection point. This is not obvious from the above definition.

\begin{lemma}\label{condition of tangential}
  Let \{$\mathcal{F}, \mathcal{G}$\} be a pair of transverse measured foliations.
  Then the horospheres $\operatorname{HS}(\mathcal{F}, s)$ and $\operatorname{HS}(\mathcal{G}, t)$ are tangent
   to each other if and only if
   \begin{equation}\label{equ:tangent}
    s \cdot  t = i(\mathcal{F}, G)^2.
    \end{equation}
   When the condition holds, $\operatorname{HS}(\mathcal{F}, s)$ and $\operatorname{HS}(\mathcal{G}, t)$
   have a unique intersection point.
\end{lemma}
\begin{proof} ($\Longleftarrow$) Assume that $s \cdot t = i(\mathcal{F}, G)^2.$
  Let $\mathbf{G}_{\mathcal{F},\mathcal{G}}$ be the Teichm\"uller geodesic determined by
  $\mathcal{F}$ and $\mathcal{G}$.
  Since the extremal length of $\F$ is a strictly monotonic function along $\mathbf{G}_{\mathcal{F},\mathcal{G}}$, there exists a unique Riemann surface $X_{s} \in \mathbf{G}_{\mathcal{F},\mathcal{G}}$ such that $\operatorname{Ext}_{X_{s}}(\mathcal{F}) = s$. Since
  $\operatorname{Ext}_{X_{s}}(\mathcal{F})\operatorname{Ext}_{X_{s}}(\mathcal{G}) = i(\mathcal{F},\mathcal{ G})^2,$
  $\operatorname{Ext}_{X_{s}}(\mathcal{G})$  must be equal to $t$.

 By our construction, $X_{s} \in \operatorname{HS}(\mathcal{F}, s) \cap \operatorname{HS}(\mathcal{G}, t)$.

 It remains to show that   $\operatorname{HB}(\mathcal{F}, s) \cap \operatorname{HB}(\mathcal{G}, t) = \emptyset$. For any point $X \in \operatorname{HB}(\mathcal{F}, s)$, we have $\operatorname{Ext}_{X}(\mathcal{F}) < s$. Using Minsky's inequality (Theorem \ref{Minsky}),
        we obtain
    $$\operatorname{Ext}_{X}(\mathcal{G}) \geq \frac{i(\mathcal{F}, \mathcal{G})^2}{\operatorname{Ext}_{X}(\mathcal{F})} > \frac{i(\mathcal{F}, \mathcal{G})^2}{s} = t.$$
    This implies that $X \notin \operatorname{HB}(\mathcal{G}, t)$. Similarly, for any point $Y \in \operatorname{HB}(\mathcal{G}, t)$, we have $Y \notin \operatorname{HB}(\mathcal{F}, s)$.

Thus $\operatorname{HS}(\mathcal{F}, s)$ is tangent to $\operatorname{HS}(\mathcal{G}, t)$ at $X_{s}$.

\medskip

 ($\Longrightarrow$) Conversely, assume that $\operatorname{HS}(\F, s)$ and $\operatorname{HS}(\G, t)$
 are tangent to each other.
 Set $r = \frac{i(\F,\G)^2}{t}$. From the above argument, the horospheres $\operatorname{HS}(\F, r)$
 and $\operatorname{HS}(\G, t)$ are tangent to each other at $X_r$.
 We claim that $r=s$.

 In fact, if $s<r$, then the closed horoball $\overline{\operatorname{HB}}(\mathcal{F}, s)$ is contained in
 $\operatorname{HB}(\mathcal{F}, r)$.
 This implies that  $\overline{\operatorname{HB}}(\mathcal{F}, s) \cap \overline{\operatorname{HB}}(\mathcal{G}, t) = \emptyset$,
 which contradicts our assumption.

 If $s> r$, then $X_r \in \operatorname{HB}(\mathcal{G}, s)$,  which contradicts with
 the fact that $X_r\in \operatorname{HS}(\mathcal{F}, r) \cap \operatorname{HS}(\mathcal{G}, t)$.

\medskip

We have proved the equivalence.  The above proof also shows that any
$X \in \operatorname{HS}(\mathcal{F}, s) \cap \operatorname{HS}(\mathcal{G}, t)$ lies on the geodesic
$\mathbf{G}_{\mathcal{F},\mathcal{G}}$. Thus $X=X_s$ is unique.
\end{proof}

The next lemma studies the question when a triple of  horospheres are tangent to each other.
For simplicity, we only consider measured foliations corresponding to simple closed curves.
This is sufficient for application in the proof of Theorem \ref{Result four}.

A pair of simple closed curves $(\alpha, \beta)\in \mathcal{S}\times \mathcal{S}$ is \emph{filling} if $i(\alpha, \gamma) + i(\beta, \gamma) > 0$
 for all  $\gamma\in \mathcal{S}$. If $(\alpha, \beta)$ is filling, there is a unique Teichm\"uller geodesic
 determined by $\alpha$ and $\beta$. We shall denote such a geodesic by $\mathbf{G}_{\alpha,\beta}$.

\begin{lemma}\label{Dense two}
  Suppose that all the pairs $(\alpha, \beta), (\alpha, \gamma), (\beta, \gamma)$ are filling.
  Then there exist unique $s,t,r\in \mathbb{R}_+$ such that the horospheres  $\operatorname{HS}(\alpha, r)$, $\operatorname{HS}(\beta, s)$ and $\operatorname{HS}(\gamma, t)$ are tangent to each other.
\end{lemma}
\begin{proof}
By Theorem \ref{Minsky},  any $X\in \mathbf{G}_{\alpha,\beta}$ must satisfy
\begin{equation*}
  {\textrm{Ext}_{X}(\alpha)\textrm{Ext}_{X}(\beta)}
  =i(\alpha,\beta)^2.
\end{equation*}
Let $$ k = \frac{i(\alpha,\gamma)^2}{i(\beta,\gamma)^2}.$$
There is a unique $X_{0}\in \mathbf{G}_{\alpha,\beta}$ such that
\begin{equation*}
 {\frac{\textrm{Ext}_{X_{0}}(\alpha)}
  {\textrm{Ext}_{X_{0}}(\beta)}}=k.
\end{equation*}
Let $r = \textrm{Ext}_{X_{0}}(\alpha), s = \textrm{Ext}_{X_{0}}(\beta)$ and
   \begin{equation*}
     t = \frac{i(\alpha,\gamma)^2}{r} = \frac{i(\beta,\gamma)^2}{s}.
   \end{equation*}
According to Lemma \ref{condition of tangential},
the horospheres  $\operatorname{HS}(\alpha, r), \operatorname{HS}(\beta, s)$ and
$\operatorname{HS}(\gamma, t)$ are tangent to each other. Moreover, the solution
$(s, t, r)$ is unique.
\end{proof}

Triples of  horospheres tangent to each other are flexible in Teichm\"uller space:

\begin{lemma}\label{Dense one}
  Let $(\alpha, \beta)\in \mathcal{S}\times \mathcal{S}$ be filling. Given any $t\in \mathbb{R}_{+}$ and $\epsilon>0$,
  there exists a simple closed curve $\gamma\in \mathcal{S}$ such that both $(\alpha, \gamma)$ and $(\beta,\gamma)$ are filling and
  $$|\frac{i(\alpha, \gamma)}{i(\beta, \gamma)}-t|<\epsilon.$$
  \end{lemma}
\begin{proof}
It is not hard to show that the map  $$\frac{i(\alpha, \cdot)}{i(\beta, \cdot)} : \mathcal{PMF}\to \mathbb{R}_{\geq 0} \cup \{+\infty\}$$ is continuous. Since $\mathcal{PMF}$ is path-connected and $$\frac{i(\alpha, \alpha)}{i(\beta, \alpha)}=0,
  \frac{i(\alpha, \beta)}{i(\beta, \beta)}=+\infty,$$
  we have for any $t\in \mathbb{R}_{+}$,
  there is a measured foliation $\mathcal{F}$ such that
 $$ \frac{i(\alpha, \mathcal{F})}{i(\beta, \mathcal{F})}=t.$$ By the density of uniquely ergodic measured foliations,
  we may assume that $\mathcal{F}$ is uniquely ergodic and
  $$|\frac{i(\alpha, \mathcal{F})}{i(\beta, \mathcal{F})}-t|< \frac{\epsilon}{2}.$$

It is obvious that $(\alpha, \mathcal{F})$ (and also $(\beta, \mathcal{F})$) are transverse.
If we choose $\gamma\in \mathcal{S}$ sufficiently close to $[\mathcal{F}]$ in $\mathcal{PMF}$,
then both $(\alpha, \gamma), (\beta,\gamma)$ are filling and
  $$|\frac{i(\alpha, \gamma)}{i(\beta, \gamma)}-t|<\epsilon.$$
\end{proof}

\section{Horosphere-preserving diffeomorphisms}\label{sec:diff}
This section contains the proof of Proposition \ref{Result two} and Theorem \ref{Result four},
as stated in \S \ref{sec:into}.
 Throughout this section,
 $f: \mathcal{T}_{g,n} \rightarrow \mathcal{T}_{g,n}$ will denote a diffeomorphism that preserves horospheres.

\subsection{Proof of Proposition \ref{Result two}}
We first prove:
\begin{lemma}\label{Inverse_is_true}
The inverse map $f^{-1}$ also preserves horospheres.
\end{lemma}
\begin{proof}
Assume that $f(X)=Y$.
Let $W=\langle \mu_{1}, ..., \mu_{6g - 7 + 2n}\rangle$ be the tangent space of $\operatorname{HS}(\mathcal{F}, Y)$ at $Y$.
 Then the pull-back of $W$ by $f$
  $$ V=\langle f^{*}(\mu_{1}), ..., f^{*}(\mu_{6g - 7 + 2n})\rangle $$
  is a linear subspace of $T_{X}\mathcal{T}_{g,n}$ of
   dimension $6g - 7 + 2n$.

There is a quadratic differential $q \in Q(X)$ such that $V$ is the tangent space of $\operatorname{HS}(\mathcal{F}_{v}(q), X)$ and $\operatorname{HS}(\mathcal{F}_{h}(q), X)$ at $X$ (see Lemma \ref{lemma:gr}).
Since $f$ preserves horospheres, it maps $\operatorname{HS}(\mathcal{F}_{v}(q), X)$ and $\operatorname{HS}(\mathcal{F}_{h}(q), X)$ to two horospheres, denoted by
$\operatorname{HS}(\mathcal{F}',Y)$ and $\operatorname{HS}(\G', Y)$.

Note that $\operatorname{HS}(\mathcal{F}',Y)$ and $\operatorname{HS}(\G', Y)$ are tangent to each other at
$Y$, and their tangent space at $Y$ is equal to $W$.
According to Lemma \ref{necessary condition of tangent}, we have
  $\mathcal{F} = k \mathcal{F}'$ or $\mathcal{F} = k \G'$
  for some $k \in \mathbb{R}_{+}$. Thus (see Remark \ref{coincide})
  $\operatorname{HS}(\mathcal{F}, Y) = f[\operatorname{HS}(\mathcal{F}_{v}(q), X)]$
  or
  $f[\operatorname{HS}(\mathcal{F}_{h}(q), X)].$
  The proof is complete.
\end{proof}

\begin{remark}
The $C^1$-smoothness of $f$ is used in the proof of Lemma \ref{Inverse_is_true}.
If we assume that $f: \mathcal{T}_{g,n} \rightarrow \mathcal{T}_{g,n}$ is a homeomorphism and
both $f, f^{-1}$  preserve horospheres, then the results in this section are still valid.
\end{remark}

\begin{corollary}\label{MHTH}
The map $f$ preserves horoballs.
\end{corollary}
\begin{proof}
  Let $\operatorname{HS}(\mathcal{F}, s)$ be a horosphere.
  Assume that
$\operatorname{HS}(\mathcal{G}, t) = f[\operatorname{HS}(\mathcal{F}, s)].$
Note that $\T$ is separated by  $\operatorname{HS}(\mathcal{G}, t)$  into $\operatorname{HB}(\G, t)$ and $\mathcal{T}_{g,n}/\overline{\operatorname{HB}}(\G, t)$. We claim that
 $$ f[\operatorname{HB}(\mathcal{F}, s)] = \operatorname{HB}(\mathcal{G}, t), f[\mathcal{T}_{g,n}/\operatorname{HB}(\mathcal{F}, s)] = \mathcal{T}_{g,n}/\operatorname{HB}(\mathcal{G}, t).$$

Choose any $X \in \operatorname{HS}(\mathcal{F}, s)$.
There is a unique quadratic differential $q \in Q(X)$ such that $\mathcal{F}_{h}(q) = \mathcal{F}$.
  According to Lemma \ref{tangential}, we know that $\operatorname{HS}(\mathcal{F}, X)$ is tangent to $\operatorname{HS}(\mathcal{F}_{v}(q), X)$ at $X$.
  This implies that $f[\operatorname{HS}(\mathcal{F}, X)]$ is tangent to $f[\operatorname{HS}(\mathcal{F}_{v}(q), X)]$ at $f(X)$.
  In particular, $$f[\operatorname{HS}(\mathcal{F}_{v}(q), X)] \subset \mathcal{T}_{g,n}/\operatorname{HB}(\mathcal{G}, t).$$
  Since $\operatorname{HS}(\mathcal{F}_{v}(q), X) \subset \mathcal{T}_{g,n}/\operatorname{HB}(\mathcal{F}, s)$, we have $$f[\mathcal{T}_{g,n}/\operatorname{HB}(\mathcal{F}, s)] = \mathcal{T}_{g,n}/\operatorname{HB}(\mathcal{G}, t).$$

  The proof is complete.
\end{proof}

Now we prove  Proposition \ref{Result two}:

\medskip

 \emph{Let $f: \mathcal{T}_{g,n}\to \mathcal{T}_{g,n}$ be a diffeomorphism that preserves horospheres.
  Assume that $\mathcal{F}\in \mathcal{MF}_{ind}$  and $f[\operatorname{HS}(\mathcal{F}, X)] = \operatorname{HS}(\mathcal{G}, Y)$. Then
  \begin{enumerate}
    \item $ f[\operatorname{HS}(\mathcal{F}, Z)] = \operatorname{HS}(\mathcal{G}, f(Z))$ for all $Z \in \mathcal{T}_{g,n}$;
    \item $\mathcal{G}\in \mathcal{MF}_{ind}$.
  \end{enumerate}}

\begin{proof}[Proof of Proposition \ref{Result two}]

 (1)  Let $Z\in \T$. By Lemma \ref{Inverse_is_true}, we can assume that $$f^{-1}[\operatorname{HS}(\mathcal{G}, f(Z))] =
       \operatorname{HS}(\mathcal{F}', Z).$$

       Consider the case that $\operatorname{Ext}_{Y}(\mathcal{G}) \leq \operatorname{Ext}_{f(Z)}(\mathcal{G})$.
       By  Corollary \ref{MHTH}, we have
       $$f^{-1}[\operatorname{HB}(\mathcal{G},Y)]
       =\operatorname{HB}(\mathcal{F}, X), f^{-1}[\operatorname{HB}(\mathcal{G}, f(Z))]
       = \operatorname{HB}(\mathcal{F}', Z).$$
       Since $\operatorname{HB}(\mathcal{G}, Y)
       \subset \operatorname{HB}(\mathcal{G}, f(Z))$,
       we have
       $ \operatorname{HB}(\mathcal{F}, X) \subset \operatorname{HB}(\mathcal{F}', Z)$.
      It follows from Proposition \ref{Main-Lemma-One} that $\mathcal{F}' = k\mathcal{F}$ for some $k \in \mathbb{R}_{+}$. It follows that
       \begin{equation*}
         \operatorname{HS}(\mathcal{F}, Z) = \operatorname{HS}(\mathcal{F}', Z)\ \ \textrm{and}\ \
         f[\operatorname{HS}(\mathcal{F}, Z)] =  \operatorname{HS}(\mathcal{G}, f(Z)).
       \end{equation*}

       The case that $\operatorname{Ext}_{Y}(\mathcal{G}) \geq \operatorname{Ext}_{f(Z)}(\mathcal{G})$ can be proved in the same way.

\bigskip

  (2) Suppose not,  $\mathcal{G}$ has more than one indecomposable component.
  Let $\mathcal{G}_0$ be an indecomposable component of ${\mathcal{G}}$.
  According to Corollary \ref{Relations},
  there exists $Y_0 \in \mathcal{T}_{g,n}$ such that
    $$ \operatorname{HB}(\mathcal{G}, Y) \subset \operatorname{HB}(\mathcal{G}_0, Y_0).$$
  Assume that $f^{-1}[\operatorname{HB}(\mathcal{G}_0, Y_0)] = \operatorname{HB}(\F_0, X_0)$, where $X_0 = f^{-1}(Y_0)$.  Then we have
    $$ \operatorname{HB}(\mathcal{F}, X) \subset \operatorname{HB}(\F_0, X_0).$$
  Applying Proposition \ref{Main-Lemma-One} to $\mathcal{F}$, which is assumed to be indecomposable,
  we have $\F_0 = k \cdot \mathcal{F}$ for some $k \in \mathbb{R}_{+}$. This implies that
    $$ \operatorname{HS}(\mathcal{G}, Y_0) = f[\operatorname{HS}(\F_0, X_0)] = \operatorname{HS}(\mathcal{G}_0, Y_0).$$
  This leads to a contradiction, since $\mathcal{G}_0 \neq k \mathcal{G}$ for any $k \in \mathbb{R}_{+}$.
\end{proof}

\begin{corollary}\label{bijection}
   Let $\mathcal{F} \in \mathcal{MF}_{ind}$ and  $s \in \mathbb{R}_{+}$. Then there exists $\mathcal{G} \in \mathcal{MF}_{ind}$ such that $f[\operatorname{HS}(\mathcal{F}, s)] = \operatorname{HS}(\mathcal{G}, t(s))$, where $t(s)$ is a strictly monotonically increasing function of $s$.
\end{corollary}

\subsection{The map $f$ induces an automorphism of $\mathcal{C}(\mathcal{S})$}
As above, let $$f: \mathcal{T}_{g,n} \rightarrow \mathcal{T}_{g,n}$$ be a diffeomorphism that preserves horospheres.
It follows from Proposition \ref{Result two} and Remark \ref{coincide} that $f$ induces a bijection
$$f_{*}: \mathcal{MF}_{ind} \rightarrow \mathcal{MF}_{ind}.$$
Moreover, $f_*$ maps projective equivalence classes to projective equivalence classes.

Denote by $\mathcal{UMF}$ be the set of uniquely ergodic measured foliations. It is clear that $\mathcal{S}$ and $\mathcal{UMF}$ are contained in $\mathcal{MF}_{ind}$.

  For $\mathcal{F} \in \mathcal{MF}_{ind}$, we denote
    $$\mathcal{N}(\mathcal{F}) = \{\mathcal{G} \in \mathcal{MF}_{ind}\ |\ i(\mathcal{F}, \mathcal{G}) = 0\}.$$

 Two measured foliations $\mathcal{F}$ and $\mathcal{G}$ are \emph{topologically equivalent}
if they (considered without their transverse measures) are isotopic up to Whitehead moves.
\begin{lemma}\label{topological equivalent}
  Let $\mathcal{F},\mathcal{G} \in \mathcal{MF}_{ind}$. Then $\mathcal{N}(\mathcal{F}) = \mathcal{N}(G)$ if and only if $\mathcal{F}$ and $\mathcal{G}$ are topologically equivalent. We have $\mathcal{N}(\mathcal{F}) = \{k \cdot \mathcal{F} \ | \ k\in \mathbb{R}_+\}$ if and only if $\mathcal{F} \in \mathcal{UMF}$.
\end{lemma}
Lemma \ref{topological equivalent} was proved in \cite[Theorem 4.1]{ivanov2001isometries}.

\begin{proposition}\label{preserves indecomposable}
 The map $f_{*}: \mathcal{MF}_{ind} \rightarrow \mathcal{MF}_{ind}$ satisfies:
$$i(f_{*}(\mathcal{F}), f_{*}(\mathcal{G})) = 0 \Longleftrightarrow i(\mathcal{F}, \mathcal{G}) = 0.$$
\end{proposition}
\begin{proof}
Let $\mathcal{F},\mathcal{G} \in \mathcal{MF}_{ind}$ with $i(\mathcal{F},\mathcal{G}) = 0$.
 We define $\F+\G$ be the measured foliation equivalent to $i(\F,\cdot)+i(\G,\cdot)$ (when $\F$ and $\G$ are multi-curves,
 $\F+\G$ is the union of the curves).

Take a sequence of $X_k\in \mathcal{T}_{g,n}$
such that $\mathrm{Ext}_{X_k}(\mathcal{F}+\mathcal{G})\to 0$ as $k\to \infty$.
By the monotonicity of extremal length,
 we have $$\mathrm{Ext}_{X_k}(\mathcal{F})\leq \mathrm{Ext}_{X_k}(\mathcal{F}+\mathcal{G}),
\mathrm{Ext}_{X_k}(\mathcal{G})\leq \mathrm{Ext}_{X_k}(\mathcal{F}+\mathcal{G}).$$ Then
$$\mathrm{Ext}_{X_k}(\mathcal{F})\to 0, \mathrm{Ext}_{X_k}(\mathcal{G})\to 0$$  as $k\to \infty$.
Equivalently, we have
$$X_k\in \operatorname{HS}(\mathcal{F}, s_k)\cap \operatorname{HS}(\mathcal{G}, t_k)$$ with $s_k,t_k\to 0$
as $k\to \infty$. By Corollary \ref{bijection}, $Y_k:=f(X_k)$ belongs to
\begin{eqnarray*}
f\left(\operatorname{HS}(\mathcal{F}, s_k)\right) \cap f\left((\operatorname{HS}(\mathcal{G}, t_k)\right)
:= \operatorname{HS}(\mathcal{F}', s'_k) \cap \operatorname{HS}(\mathcal{G}', t'_k),
\end{eqnarray*}
with $s_k',t_k'\to 0$ as $k\to \infty$. This implies that $i(\mathcal{F}',\mathcal{G}')=0$. Otherwise,
the product $\mathrm{Ext}_{Y_k}(\mathcal{F}')\mathrm{Ext}_{Y_k}(\mathcal{G}')$ is bounded below by
$i(\mathcal{F}',\mathcal{G}')^2$, which is impossible.
\end{proof}

  Combining Lemma \ref{topological equivalent} with Proposition
  \ref{preserves indecomposable}, we obtain:
\begin{corollary}\label{preserve uniquely ergodic}
 The map $f_{*}$ satisfies $f_*(\mathcal{UMF})=\mathcal{UMF}$.
\end{corollary}

\begin{proposition}\label{automorphism of curve complex}
  The map $f$ gives rise to an automorphism of the complex of curves $f_{*}: \mathcal{C}(S) \rightarrow \mathcal{C}(S)$.
\end{proposition}
\begin{proof}
  It suffices to prove that $f_{*}(\gamma) \in \mathcal{S}$ when $\gamma \in \mathcal{S}$.
  According to Proposition \ref{preserves indecomposable},
  we can assume that $\mathcal{G} = f_{*}(\gamma) \in \mathcal{MF}_{ind}$.
  Denote by $\widetilde{\mathcal{G}}$ the unmeasured foliation obtained from $\mathcal{G}$ by forgetting the measure.

  First, we observe that the dimension of the space of transverse measures on $\widetilde{\mathcal{G}}$ is one.
  If not, there exists some other $\mathcal{G}' \in \mathcal{MF}_{ind}$ which is topologically equivalent to $\mathcal{G}$, but not projectively equivalent. Denote
    $\mathcal{F} = f_{*}^{-1}(\mathcal{G}')\in \mathcal{MF}_{ind}.$
 By Proposition \ref{preserves indecomposable}, we have
    $$ \mathcal{N}(\gamma) = \mathcal{N} (f_{*}^{-1}(\mathcal{G}))= \mathcal{N} (f_{*}^{-1}(\mathcal{G}')) = \mathcal{N}(\mathcal{F}),$$
  since  ${\mathcal{N}(\mathcal{G}) = \mathcal{N}(\mathcal{G}')}$. Applying Lemma \ref{topological equivalent} we conclude that $\mathcal{F}$ and $\gamma$ are projectively equivalent. This is a contradiction to the assumption that
  $\G$ and $\G'$ are not projective equivalent.

There are three possibilities:

 (i) $\G \in \mathcal{S}$. This is what we want to prove.

 (ii) $\mathcal{G}$ is a uniquely ergodic measured foliation on $S$.
  This can not happen because $\gamma= f_{*}^{-1}(\mathcal{G})\notin \mathcal{UMF}$.

 (iii) The remaining case is that $\G$ is uniquely ergodic on $X_{0}$,
  which is a proper subsurface of $S$. Let $\beta$ be a boundary component of $X_0$.
  Denote $\F=f_*^{-1}(\beta)$. Then $\F\in \MFind$ is either a simple closed curve or a
  minimal ergodic component.  In both cases, there always exists  $\alpha \in \mathcal{S}$ such that $i(\mathcal{F}, \alpha) \neq 0$ and $i(\gamma, \alpha) = 0$. Then we have
  $$i(\beta, f_*(\alpha)) \neq 0, i(\G, f_*(\alpha)) = 0.$$
  Due to our construction, any measured foliation that intersects with the boundary component $\beta$ must
 also intersects with $\G$ (since the measured foliation must cross the collar neighborhood of $\beta$ and
  $\G$ is filling on the subsurface $X_0$).
   This leads to a contraction.
\end{proof}

\subsection{Proof of Theorem \ref{Result four}}
The proof is motivated by \cite[\S 5]{ivanov2001isometries}.

\begin{proof}[Proof of Theorem \ref{Result four}]
Let $f: \mathcal{T}_{g,n} \rightarrow \mathcal{T}_{g,n}$ be a diffeomorphism that preserves horospheres.
We show that $f$ is induced by an element of the extended mapping class group.

  By Theorem \ref{automorphism of curve complex},
  $f$ induces an automorphism $f_{*}$ of the complex of curve $\mathcal{C}(S)$.
  Then the theorem of Ivanov (Theorem \ref{Th B}) implies that $f_{*}$ acts on $\mathcal{C}(S)$ as an element $\phi$ of the extended mapping class group.
   Replacing $f$ by $\phi^{-1} \circ f$, we can assume that $f_{*} =\mathrm{id}: \mathcal{S}\to \mathcal{S}$.
 It remains to prove that $f = \mathrm{id}$ on $\mathcal{T}_{g,n}$.

Let $(\alpha, \beta)$ be a pair of filling simple closed curves.
  Denote by $\mathbf{G}_{\alpha, \beta}$ the Teichm\"uller geodesic determined by $\alpha$ and $\beta$.
   We first claim that $$f(\mathbf{G}_{\alpha, \beta}) = \mathbf{G}_{\alpha, \beta}.$$
   In fact, for any $X \in \mathbf{G}_{\alpha, \beta}$,  $\operatorname{HS}(\alpha, X)$ is tangent to  $\operatorname{HS}(\beta, X)$ at $X$
  (see Lemma \ref{condition of tangential}).
 Since $f_*(\alpha)=\alpha, f_*(\beta)=\beta$ and $f$ preserves horospheres,
  $\operatorname{HS}(\alpha, f(X))$ is tangent to  $\operatorname{HS}(\beta, f(X))$ at $f(X)$.
 Thus ${f(X) \in \mathbf{G}_{\alpha, \beta}}$. This shows that $f$ preserves the
 Teichm\"uller geodesic $\mathbf{G}_{\alpha, \beta}$.

  We next show that $f$ is identity on $\mathbf{G}_{\alpha, \beta}$.
  We parameterize $\mathbf{G}_{\alpha, \beta}$ by $t \to \mathbf{G}_{\alpha, \beta}(t)$ such that
  the pair of measured foliation $(\F,\G)$ is changed to $(e^t\alpha, e^{-t}\beta)$. Without loss of generality,
  we may also assume that
  \begin{equation*}
  \frac{\textrm{Ext}_{\mathbf{G}_{\alpha, \beta}(0)}(\alpha)}
  {\textrm{Ext}_{\mathbf{G}_{\alpha, \beta}(0)}(\beta)}=1.
\end{equation*}
  According to Lemma \ref{Dense one}, for each $t$, there exists a simple closed curve $\gamma$
  such that both $(\alpha, \gamma)$ and $(\beta,\gamma)$ are filling, and
  $$k=\frac{i(\alpha,\gamma)^2}{i(\beta,\gamma)^2}\approx e^{2t}.$$
 There is a unique $X\in \mathbf{G}_{\alpha,\beta}$ such that
$$
 {\frac{\textrm{Ext}_{X}(\alpha)}
  {\textrm{Ext}_{X}(\beta)}}=k.
$$
Let $r = \textrm{Ext}_{X}(\alpha), s = \textrm{Ext}_{X}(\beta)$.
Then the proof of Lemma \ref{Dense two} shows that the horospheres  $\operatorname{HS}(\alpha, r), \operatorname{HS}(\beta, s)$ are tangent to the horosphere $\operatorname{HS}(\gamma,  \frac{i(\alpha,\gamma)^2}{r})$.
 Note that $X$ is close to $\mathbf{G}_{\alpha,\beta}(t)$.

 As a result, there is a dense subset $\mathbf{G} \subset \mathbf{G}_{\alpha, \beta}$
  such that for every $X \in \mathbf{G}$, there exist a simple closed curve $\gamma$ and $r,s,t \in \mathbb{R}_{+}$ such that horospheres $\operatorname{HS}(\gamma, r)$, $\operatorname{HS}(\alpha, s)$, $\operatorname{HS}(\beta, t)$ are tangent to each other, and $X$ is the tangent point of $\operatorname{HS}(\alpha, s)$ and $\operatorname{HS}(\beta, t)$.
The images of the above triple of horospheres under $f$ are also horospheres tangent to each other
 (still associated with $\alpha,\beta,\gamma$).
 As we have observed in Lemma \ref{Dense two}, $X$ is the unique solution of the tangent problem. Thus
 $f(X)=X$.

  Since the set of Teichm\"uller geodesics determined by filling pairs of simple closed geodesics is dense
  in $\T$ \cite{Masur-dense},
  it follows from continuity that $f=\mathrm{id}$.
\end{proof}

\section{Metric balls and Busemann functions}\label{sec:Busemann}

  For $X \in \mathcal{T}_{g,n}$ and $r \in \mathbb{R}_{+}$, we denote by
    \begin{equation*}
      \operatorname{B}(X, r) = \{ Y \in \mathcal{T}_{g,n}\ |\ d_{\mathcal{T}}(X, Y) < r\}
    \end{equation*}
  the open metric ball of radius $r$ centered at $X$.
  The closure of $\operatorname{B}(X, r)$ is called a closed metric ball.
 It is not hard to show that (see \cite[Lemma 3.2]{bourque2016non})

\begin{lemma}
  Every closed metric ball in $\mathcal{T}_{g,n}$ is a countable intersection of closed horoballs.
\end{lemma}

In this section, we prove the following converse result.  Recall that $\mathcal{MF}_{ind}$ is the set of indecomposable measured foliations on $S$.
\begin{lemma}\label{HB-MB}
  Let $\mathcal{F} \in \mathcal{MF}_{ind}$.
   Then every open horoball in $\mathcal{T}_{g,n}$ associated to $\mathcal{F}$ is a nested union of open metric balls.
\end{lemma}
  The about lemma was proved in special cases, by Bourque and Rafi \cite{bourque2016non}
  in the case that $\mathcal{F}$ is a simple closed curve,
  and by Masur \cite{Masur1981Transitivity} in the case that $\mathcal{F}$ is uniquely ergodic.

Our proof is to show that horospheres associated with indecomposable measured foliations
are level sets of Busemann functions.

 \subsection{Busemann functions}
  Let $\mathcal{F} \in \mathcal{MF}$. Denote by
  $\mathbf{G}(\cdot) : [0,\infty)\rightarrow\mathcal{T}_{g,n}$
  the Teichm\"uller geodesic ray determined by $\mathcal{F}$ and $X_0 = \mathbf{G}(0)$.

  \begin{definition}(Busemann function)
  With above notation, the Busemann function associated to $\mathbf{G}(\cdot)$ is the map $B_{\mathbf{G}}: \mathcal{T}_{g,n} \rightarrow \mathbb{R}$
  defined by
    \begin{equation*}
      B_{\mathbf{G}}(\cdot) = \lim_{t \rightarrow \infty}
      \left(d_{\mathcal{T}}(\cdot, \mathbf{G}(t)) -
      d_{\mathcal{T}}(\mathbf{G}(0),\mathbf{G}(t))\right).
    \end{equation*}
\end{definition}
To see the convergence, let $X\in \T$ and denote
\begin{eqnarray*}
D(t)&=& d_{\mathcal{T}}(X, \mathbf{G}(t)) - d_{\mathcal{T}}(\mathbf{G}(0),\mathbf{G}(t))\\
&=& d_{\mathcal{T}}(X, \mathbf{G}(t))-t.
\end{eqnarray*}
We observe that $D(t)$ is a bounded non-increasing function.
In fact, $D(t)\geq -d(\mathbf{G}(0),X)$. For any $0\leq t_1\leq t_2$, we have
\begin{eqnarray*}
D(t_2)&=& d_{\mathcal{T}}(X, \mathbf{G}(t_2)) - t_2\\
&=& d_{\mathcal{T}}(X, \mathbf{G}(t_2))- d_{\mathcal{T}}(\mathbf{G}(t_2), \mathbf{G}(t_1))-t_1 \\
&\leq& d_{\mathcal{T}}(X, \mathbf{G}(t_1))-t_1 \\
&=& D(t_1).
\end{eqnarray*}

For simplicity, we denote $B = B_{\mathbf{G}}$. Let
 $$\operatorname{L}(B, s) = \{X \in \mathcal{T}_{g,n}\ |\ B(X) = s\}$$
and
$$ \textrm{SL}(B, s) = \{X \in \mathcal{T}_{g,n}\ |\ B(X) < s\},$$
denote the level set and sub-level set of $B$, respectively. Note that $X_0\in \operatorname{L}(B, 0)$. Let
\begin{equation*}
\operatorname{S}(X, \varepsilon) = \{Y \in \mathcal{T}_{g,n}\ |\ d_{\mathcal{T}}(X, Y) = \varepsilon\}
\end{equation*}
and
\begin{equation*}
      \operatorname{B}(X, \varepsilon) = \{Y \in \mathcal{T}_{g,n}\ |\ d_{\mathcal{T}}(X, Y) < \varepsilon\}
    \end{equation*}
be the metric sphere and the metric ball with center $X \in \mathcal{T}_{g,n}$ and radius $\varepsilon \in \mathbb{R}_{+}$, respectively.

\begin{definition}
  Let $\{\mathbf{M}_{n}\}$ be a sequence of non-empty subsets of $\mathcal{T}_{g,n}$.
  We define the upper and lower limits of the sequence as follows:
\begin{enumerate}
    \item The upper limit $\varlimsup \mathbf{M}_{n}$ consists of all
          accumulation points of any sequences $\{X_{n}\}$
          with $X_{n} \in \mathbf{M}_{n}$. Thus $X \in \varlimsup \mathbf{M}_{n}$
          if and only if each $\operatorname{S}(X, \varepsilon)$, $\varepsilon > 0$,
          intersects with infinitely many $\mathbf{M}_{n}$.
    \item The lower limit $\varliminf \mathbf{M}_{n}$ consists of all points $X$
          such that each $\operatorname{S}(X, \varepsilon)$, $\varepsilon > 0$,
          intersects with all but a finite number of $\mathbf{M}_{n}$.
  \end{enumerate}
By definition, $\varliminf \mathbf{M}_{n} \subset \varlimsup \mathbf{M}_{n}$.
If $\varliminf \mathbf{M}_{n} =  \varlimsup \mathbf{M}_{n}$,
we denote by $\lim  \mathbf{M}_{n}$.
\end{definition}

The next lemma describes the relationship between metric spheres and level sets of Busemann functions.
The proof is obtained in \cite{busemann2005geometry},
which applies to a general geodesic metric spaces. We give the proof here for convenience of the readers.

\begin{lemma}\label{Buse-Limit-Sphere}
Let $\operatorname{S}(\mathbf{G}(t),t)$ and $\operatorname{B}(\mathbf{G}(t), t)$, respectively,  be the metric sphere and
 the metric ball with center at $\mathbf{G}(t)$
and passing through $X_0$. Then
$$ \lim_{t \rightarrow \infty} \operatorname{S}(\mathbf{G}(t), t)
      = \mathrm{L}(B, 0)\ \ and\ \ \lim_{t \rightarrow \infty} \operatorname{B}(\mathbf{G}(t), t)
      = \mathrm{SL}(B, 0).$$
\end{lemma}
\begin{proof}
We first show that $\varlimsup \textrm{S}(\mathbf{G}(t), t) \subset \operatorname{L}(B, 0)$.

Let $Y$ be an arbitrary accumulation point, that is,
there is a sequence $X_{n} \in \textrm{S}(\mathbf{G}(t_n), t_n)$ such that $X_{n} \rightarrow Y$ as $t_n\to \infty$.
 It suffices to show that $B(Y)=0$.
  In fact, by definition,
\begin{align*}
           B(Y) & = \lim_{n \rightarrow \infty} \{d_{\mathcal{T}}(Y, \mathbf{G}(t_{n}))
                      - t_{n}\} \\
                    & = \lim_{n \rightarrow \infty}
                        \{d_{\mathcal{T}}(X_n, \mathbf{G}(t_{n})) - t_n \}\\
                    & = \lim_{n\rightarrow \infty}
                        \{d_{\mathcal{T}}(X_0, \mathbf{G}(t_{n})) - t_n\}\\
                    & = 0.
         \end{align*}

It remains to prove that $\operatorname{L}(B, 0) \subset \varliminf \textrm{S}(\mathbf{G}(t), t)$.

Choose any $Y \in \operatorname{L}(B, 0)$.
As we have noted above, the Busemann function $B(\cdot)$ is the limit of a non-increasing sequence of  functions
$d_{\mathcal{T}}(\cdot, \mathbf{G}(t))-t$. It follows that
$$d_{\mathcal{T}}(Y, \mathbf{G}(t))-t\geq B(Y)=0.$$
Thus the distance between $Y$ and $\mathbf{G}(t)$ is greater than $t$.
Consider the geodesic segment $\overline{Y\mathbf{G}(t)}$ connecting $Y$ and $\mathbf{G}(t)$.
It intersects with the metric sphere $\textrm{S}(\mathbf{G}(t), X)$ at some point $X_{t}$. Then
         \begin{align*}
           d_{\mathcal{T}}(Y, X_{t}) & = d_{\mathcal{T}}(Y, \mathbf{G}(t)) - d_{\mathcal{T}}(X_{t}, \mathbf{G}(t)) \\
                           & = d_{\mathcal{T}}(Y, \mathbf{G}(t)) - d_{\mathcal{T}}(X_0, \mathbf{G}(t)) \\
                           & \to B(Y) - B(X_0) \\
                           & = 0
         \end{align*}
       This implies that when $t$ large enough, we have
         \begin{equation*}
           B(Y, \varepsilon) \cap \textrm{S}(\mathbf{G}(t), t) \neq \emptyset,
         \end{equation*}
       and then
         \begin{equation*}
           Y\in \varliminf \textrm{S}(\mathbf{G}(t), t).
         \end{equation*}
The proof of the sub-level set is similar.
\end{proof}

\subsection{Formula of Busemann functions}\label{sec:walsh}
There is an explicit formula of $B_\mathbf{G}$, due to Walsh \cite{Walsh2012The}.
As before, we denote by
  $\mathbf{G}(\cdot)$
  the Teichm\"uller geodesic ray determined by $\mathcal{F}$ and $X_{0} = \mathbf{G}(0)$.
  Denote by $\G$ the horizontal foliation of the Hubbard-Masur differential of $\F$ on $X_0$.

\begin{theorem}[Corollary 1.2 of \cite{Walsh2012The}]
Let $\F=\sum_j \F_j$ be the ergodic decomposition of $\F$. Then the Busemann function $B_{\mathbf{G}}(\cdot)$
is given by
    \begin{equation*}
     B_{\mathbf{G}}(X) =\frac{1}{2} \log\sup_{\gamma\in\mathcal{S}} \frac{E_{\F}(\gamma)}{\operatorname{Ext}_{X}(\gamma)}
     -\frac{1}{2}\log\sup_{\gamma\in\mathcal{S}} \frac{E_{\F}(\gamma)}{\operatorname{Ext}_{X_0}(\gamma)},
    \end{equation*}
where $E_{\F}(\gamma)$ is defined by
$$E_{\F}(\gamma)= \sum_j \frac{i(\F_j,\gamma)^2}{i(\F_j,\G)}.$$
\end{theorem}

By Walsh's formula,
if the measured foliation $\F\in \MFind$, then
    \begin{eqnarray*}
     B_{\mathbf{G}}(X) &=& \frac{1}{2}\log\sup_{\gamma\in\mathcal{S}} \frac{i(\F, \gamma)^2}{\operatorname{Ext}_{X}(\gamma)}
     -\frac{1}{2}\log\sup_{\gamma\in\mathcal{S}} \frac{i(\F, \gamma)^2}{\operatorname{Ext}_{X_0}(\gamma)} \\
     &=& \frac{1}{2}\log \operatorname{Ext}_{X}(\F)-\frac{1}{2}\log \operatorname{Ext}_{X_0}(\F).
    \end{eqnarray*}
The last equality holds because of Corollary \ref{coro:ext}.
Hence we obtain the following:
\begin{proposition}\label{pro:HS-LS}
Horospheres in $\T$ associated with $\F\in \MFind$ are level sets of the corresponding Busemann functions.
\end{proposition}

\begin{remark}
We can prove that a horosphere associated with a measured foliation $\F$ is the level
of a Busemann function if and only if $\F$ is indecomposable \cite{T}.
\end{remark}

\subsection{Proof of Lemma \ref{HB-MB}}
It is a direct corollary of Lemma \ref{Buse-Limit-Sphere} and Proposition \ref{pro:HS-LS}.

\section{A proof of Royden's Theorem}\label{sec:Royden}
Denote by $\mathcal{B}$ the set of Busemann functions on $\T$.
For $X\in \T$ and $\F\in \MF$, we denote by $B(X,\F)$  the Busemann function
of the Teichm\"uller geodesic ray $\mathbf{G}(t)$ determined by $\mathcal{F}$ and $X = \mathbf{G}(0)$.
The level set  $\{ Z \in \T\ |\ B(X,\F)(Z)=0\}$ will be denoted by $\mathrm{L}(X,\F)$.

Let $f:\T \to \T$ be an isometry of the Teichm\"uller metric.
Since $f$ maps Teichm\"uller geodesic rays to Teichm\"uller geodesic rays,
$f$ defines a transformation $f_*:\mathcal{B}\to \mathcal{B}$.
Given $X\in \T$ and $\F\in \MF$, we denote
$$f_*(B(X,\F))=B(Y, \G),$$
where
$Y=f(X)\in \T$ and $\G\in \MF$.
It projects to map  from $\MF$ to itself, which is still denoted by $f_*$.

\begin{proof}[Proof of Royden's Theorem]
Using the proof of Theorem \ref{Result four}, it suffices to show that
$f$ preserves horospheres determined by indecomposable measured foliations.

\bigskip

\textbf{Step 1:} $f$ preserves level sets of Busemann functions.

In fact, $f$ maps Teichm\"uller geodesic rays to Teichm\"uller geodesic rays,
and $f$ maps metric spheres to metric spheres. Thus by Lemma \ref{Buse-Limit-Sphere},
we have
$$f(\mathrm{L}(X,\F))=\mathrm{L}(Y,\G).$$
Note that $f$ also preserves sub-level sets of Busemann functions.

\bigskip

\textbf{Step 2:} If $\F\in \MFind$, then $\G\in \MFind$.

If not, $\G\notin \MFind$. Let $\G=\sum \G_j$ be its ergodic decomposition.
We claim:

\begin{lemma}\label{lemma:SL}
For all $\G_{j}$ in the ergodic decomposition of $\G$,
The sub-level set $\mathrm{SL}(Y,\G):=\{Z \in \mathcal{T}_{g,n}\ |\ B(Y,\G)(Z)< 0\}$ is contained in the horoball
$\operatorname{HB}(\G_{j}, s)$ for some $s>0$.
\end{lemma}
 \begin{proof}
To prove the claim, we use Walsh's formula for the Busemann function.
Up to an additive constant, $B(Y,\G)(\cdot)$ is of the form
\begin{equation*}
     B(Y,\G)(Z) = \frac{1}{2}\log\sup_{\gamma\in\mathcal{S}} \frac{\sum_k c_k i(\G_k,\gamma)^2}{\operatorname{Ext}_{Z}(\gamma)}.
    \end{equation*}
Thus, up to an additive constant,
$$B(Y,\G)(Z)\geq \frac{1}{2}\log\sup_{\gamma\in\mathcal{S}} \frac{ i(\G_j,\gamma)^2}{\operatorname{Ext}_{Z}(\gamma)}-\log \sqrt{c_j}.$$
 As we have observed in \S \ref{sec:walsh}, by
 Corollary \ref{coro:ext},
  $$\frac{1}{2}\log\sup_{\gamma\in\mathcal{S}} \frac{ i(\G_j,\gamma)^2}{\operatorname{Ext}_{Z}(\gamma)}= \frac{1}{2} \log \mathrm{Ext}_Z(\G_j).$$
  Thus $\mathrm{SL}(Y,\G)$ is contained in some horoball of $\G_j$.
  \end{proof}

Since $\F, \G_{j} \in\MFind$, for any $Z \in \T$, we have
$$\mathrm{L}(Z,\F)=\mathrm{HS}(\F,Z)$$
and
$$\mathrm{L}(Z,\G_{j})=\mathrm{HS}(\G_{j},Z),\ \ \mathrm{SL}(Z,\G_{j})=\mathrm{HB}(\G_{j}, Z).$$
By the above lemma, there is some $Z_{0} \in \T$ such that
$$\mathrm{HB}(\F,X) = f^{-1}(\textrm{SL}(Y, \G))\subset f^{-1}(\mathrm{HB}(\G_j,Z_{0})).$$
Let $\F_j$ be a measured foliation such that
$$\mathrm{SL}(f^{-1}(Z_{0}),\F_j)
= f^{-1}(\mathrm{HB}(\G_j,Z_{0})).$$
Then we have
$$\mathrm{HB}(\F,X)\subset \mathrm{SL}(f^{-1}(Z_{0}),\F_j).$$
Apply Lemma \ref{lemma:SL} again, for each $\F_j$,
$\mathrm{SL}(f^{-1}(Z_{0}),\F_j)$ must contained in some horoballs
associated to each ergodic component of $\F_j$. It follows from Proposition \ref{Main-Lemma-One}
that any such ergodic component of $\F_{j}$ is projectively equivalent to $\F$. It turns out that
each $\F_j$ is projectively equivalent to $\F$. And then all the $\G_j$ are projectively equivalent to each other.
This leads to a contradiction with the assumption that $\G\notin \MFind$.

\bigskip

By \textbf{Step 2}, we have shown that $f$ preserves horospheres of indecomposable measured foliations.
Using the proof of Theorem \ref{Result four}, we conclude that $f$ is an element of the extended mapping class group.
The proof is complete.
\end{proof}

\begin{remark}
We can apply the decomposition of measured foliations and Walsh's formula to study
the action of an isometry on Teichm\"uller geodesics, and then give another proof of Royden's Theorem \cite{T}.
\end{remark}

\bibliographystyle{siam}

\end{document}